\documentclass[11pt,a4paper]{article}

\usepackage[left=1.1in, right=1.1in, top=1.0in, bottom=1.0in]{geometry}
\usepackage{amsfonts,amstext,amsmath,amssymb,amsopn,amsthm}
\usepackage{mathrsfs}
\usepackage{mathtools}
\usepackage{physics}
\usepackage{algorithm, float}
\usepackage{algorithmic}
\usepackage{todonotes}
\usepackage{enumerate}
\usepackage{subcaption}
\usepackage[normalem]{ulem}
\usepackage{cite}
\usepackage[colorlinks=true, 
linkcolor=blue,
pdfstartview=FitH,      
breaklinks=true,        
bookmarksopen=true,     
bookmarksnumbered=true  
]{hyperref}

\newcommand{\Fcal}{{\mathcal F}}
\newcommand{\Gcal}{{\mathcal G}}

\renewcommand{\Bbb}{\mathbb{B}}

\newcommand{\Nbb}{\mathbb{N}}

\newcommand{\Rbb}{\mathbb{R}}

\newcommand{\epsilonhat}{{\widehat{\epsilon}}}

\newcommand{\omegahat}{{\widehat{\omega}}}

\newcommand{\tauhat}{{\widehat{\tau}}}

\newcommand{\Tbar}{{\overline{T}}}

\newcommand{\tbar}{{\overline{t}}}

\newcommand{\xbar}{{\overline{x}}}

\newcommand{\alphabar}{{\overline{\alpha}}}

\newcommand{\epsilonbar}{{\overline{\epsilon}}}

\newcommand{\paren}[1]{\left(#1\right)}

\newcommand{\klam}[1]{\left\{#1\right\}}

\renewcommand{\equiv}{:=}

\newcommand{\Ball}{{\Bbb}}

\newcommand{\set}[2]{\left\{#1\,\left|\,#2\right.\right\}}

\newcommand{\map}[3]{#1:\,#2\rightarrow #3\,}
\newcommand{\mmap}[3]{#1:\,#2\rightrightarrows #3\,}

\newcommand{\und}{\quad\mbox{and}\quad}

\DeclareMathOperator{\Id}{Id}

\DeclareMathOperator{\epi}{{\rm epi}}

\DeclareMathOperator{\prox}{prox}

\DeclareMathOperator{\argmin}{argmin\,}

\DeclareMathOperator{\Fix}{\mathsf{Fix}\,}

  \newtheorem{theorem}{Theorem}[section]
  
  \newtheorem{lemma}[theorem]{Lemma}
  \newtheorem{proposition}[theorem]{Proposition}
  \newtheorem{corollary}[theorem]{Corollary}

  \theoremstyle{definition}
  \newtheorem{definition}[theorem]{Definition}

  \newtheorem{remark}[theorem]{Remark}
  
  \newtheorem{assumption}[theorem]{Assumption}

\title{Quantitative Convergence of Proximal Splitting Iterations in Uniformly Convex Metric Spaces}
\author{
{D. Russell Luke} 
\thanks{Institute for Numerical and Applied Mathematics,
    University of Goettingen,
    37083 Goettingen, Germany. DRL was supported in part by 
    the Deutsche Forschungsgemeinschaft (DFG, German Research Foundation) – Project number 541767835
    and 566257456.
	\texttt{r.luke@math.uni-goettingen.de}}
\and {Mahshid Mirhashemi}
 \thanks{Institute for Numerical and Applied Mathematics,
University of G\"ottingen. MM was supported by
the Deutsche Forschungsgemeinschaft (DFG, German Research Foundation) – Project number 541767835
\texttt{mahshid.mirhashemi@uni-goettingen.de}}
}

\date{\today}

\begin{document}
\reversemarginpar
 \maketitle

 \begin{abstract}
We provide sufficient conditions for quantitative
convergence of the iterates of proximal splitting algorithms for minimizing a sum of functions on a metric space.
The theory does not assume that the functions have common minima, nor does it require vanishing proximal parameters or step sizes.
Our results are stated for general $p$-uniformly convex
spaces with curvature bounded above, and a corollary specializes the main theorem
to Hadamard spaces, where many assumptions for the more general setting can be dropped.
The theory is demonstrated with computation of Fr\'echet means in the
space of SPD matrices with the affine invariant metric (a Hadamard space) and the
sphere with the usual geodesic metric (a CAT($\kappa$) metric space).
\end{abstract}

{\small \noindent {\bfseries 2010 Mathematics Subject Classification:}
  Primary 
  47H09, 
  47H10, 
  49M37, 
  53C21  

 Secondary 
 90C25,	  
 90C30,	  
 51F99,   
 49M27,   
  }


\section{Motivation}\label{s:0}
 For a domain $G$ equipped with a uniformly convex metric $\map{d}{G\times G}{\Rbb}$, we consider the
 problem
 \begin{equation}\label{eq:gen prob}
     \argmin_{x\in G} \sum_{j=1}^m f_j(x)
 \end{equation}
 where the extended-valued functions $\map{f_j}{G}{(-\infty, +\infty]}$ are geodesicly convex.
An important case study is the computation of barycenters.  Given a collection of points $\{x_1,\dots,x_m\}\subset G$
and weights $\omega_j\in(0,1)$ with $\sum_{j=1}^m\omega_j=1$, compute
 \begin{equation}\label{eq:frechet mean}
     \argmin_{x\in G} \mathcal{F}(x)\equiv\sum_{j=1}^m \omega_j d(x,x_j)^p.
 \end{equation}
The  function $\Fcal$ is called the {\em Fréchet function} and its minimizer is the
{\em barycenter} or {\em Fréchet mean} of the points $\{x_1,\dots,x_m\}$.

This problem was brought to the attention of the continuous optimization community
by Bačák \cite{Bacak14,Bacak_book, Bacak13}, though it had already been studied in \cite{Jost1995}
and is an object of obvious interest
in statistics \cite{BhaLiz17, HuckNye24, ElGaHuTu21, Afsari11, MilOwePro15}.
Convergence has been determined in Hadamard spaces \cite[Theorem 3.4]{Bacak14}
and in complete Alexandrov spaces with curvature bounded above and finite diameter in \cite[Theorems 5.1 and 5.5]{OhtaPal15}.
Convergence results for firmly nonexpansive mappings more generally were obtained in \cite{AriLeuLop14}.
Rates of convergence of {\em solutions} to
\eqref{eq:frechet mean} as $m\to\infty$ have been
established in probability and expectation in \cite{Schoetz19}.  Rates of convergence of the value function
are implicit in the convergence proofs of \cite{Bacak14} and \cite{OhtaPal15}.
The value of the Fr\'echet function, however, is seldom of interest - it
is the $\argmin$ that carries the relevant information from a statistical point of view.
Incidentally this is also the case for the envelope function associated with the prox mapping - the value of the envelope is rarely used.
Rates of {\em asymptotic regularity} of sequences for $p$-uniformly convex spaces was obtained in \cite[Theorem 3.2, Corollary 3.3]{RuiLopNic15}
for {\em consistent} problems.
A consistent problem in the format \eqref{eq:gen prob} is one in which
$\cap_j\argmin f_j\neq\emptyset$.  This is a common assumption that is not infrequently satisfied for {\em feasibility} problems where
the functions $f_j$ are indicator functions to sets, or when $m=1$.
From asymptotic regularity of the sequences the authors of \cite{RuiLopNic15} are able to obtain convergence in the consistent case in
CAT(0) and CAT($\kappa$) spaces.  Rates of convergence of the iterates
of proximal splitting algorithms for consistent problems were obtained  in spaces with nonnegative curvature in
\cite[Theorem 30]{BLL} and
in CAT($\kappa$) spaces in \cite[Theorem 25]{LauLuk21}.  More generally, however, consistency fails for many
important problems.  The requirement of consistency does not permit extension of these results to
the barycenter problem, for example.  The algorithms in \cite{Bacak14, OhtaPal15} have vanishing step sizes or proximal parameters so that,
asymptotically, the proximal operators converge to the identity mapping and the resulting splitting algorithm is asymptotically consistent.
This makes it difficult (but not impossible) to get convergence rates on the iterates.  We are unaware of quantitative convergence results
either in Hadamard spaces or spaces with curvature bounded above for the iterates of basic algorithmic templates applied to the barycenter problem or
inconsistent instances of the more general problem \eqref{eq:gen prob}.

The situation is quite different for manifolds.  Variational analysis on Riemannian manifolds and the attendant algorithmic
development has been pioneered in \cite{Udriste94, Smith94, FerOli98, Ferreira2002ProximalPA, BenFerOli10, Oliveira08,
LiLiLioYao09, CruFerLucNem, KatzCooper80, AbsMahSep2008, LiLopMar09, ColLopMarMar12, AriLeuLop14}.

\subsection{Contribution and Outline}
The main result of this paper is Theorem \ref{t:msr convergence} which provides sufficient conditions for quantitative
convergence of the iterates of proximal splitting algorithms that can be written as fixed point iterations
of the form \eqref{eq:mapping family} below.  This is stated for general $p$-uniformly convex
spaces with curvature bounded above, and summarizes several stand-alone results
that provide sufficient conditions for the main assumptions of the theorem to hold.
Corollary \ref{t:msr convergence H} specializes the main theorem
to Hadamard spaces, where many assumptions for the more general setting of Theorem \ref{t:msr convergence} can be dropped.
Section \ref{s:numerics} presents numerical demonstrations of the  theory applied the computation of Fr\'echet means in the
space of SPD matrices with the affine invariant metric (a Hadamard space) and the sphere with the usual geodesic metric (a CAT($\kappa$) metric space).
We begin in Section \ref{s:1} with necessary background.

\section{Fundamentals}\label{s:1}
We briefly introduce the basic setting and tools for our analysis here.
More detailed treatments can be found in the references.

\subsection{Uniformly Convex Metric Spaces}
Let $(G, d)$ denote an abstract metric space with domain $G$ and metric $d$.
We restrict our attention to uniquely geodesic metric spaces, namely metric spaces
for which every pair of points is joined by a unique path of minimal distance with respect to the
metric $d$, otherwise known as a {\em geodesic}.
A point $z$ on the geodesic connecting points $x$ and $y$
satisfying $d(z,x)=\tau d(x,y)$ for $\tau =d(z,x)/d(x,y)$,
is denoted by $z=(1-\tau)x\oplus \tau y$.  Geodesics in linear spaces, for instance, are just line segments.
Following \cite{RuiLopNic15}
we focus on {\em $p$-uniformly convex spaces with parameter $c$}:
for $p\in (1,\infty)$,  a metric space $(G, d)$ is $p$-uniformly convex with constant
$c>0$  whenever it is a uniquely geodesic metric space, and
\begin{equation*}\label{e:p-ucvx}
\begin{aligned}
&(\forall \tau \in [0,1])(\forall x,y,z\in G)\\
& d(z, (1-\tau )x\oplus \tau y)^p\leq (1-\tau )d(z,x)^p+\tau d(z,y)^p - \tfrac{c}{2}\tau (1-\tau )d(x,y)^p.
\end{aligned}
\end{equation*}
Examples of $p$-uniformly convex spaces are $L^p$ spaces, CAT(0) spaces ($p=c=2$),
Hadamard spaces (complete CAT$(0)$ spaces), Hilbert spaces (linear Hadamard spaces).
Of particular interest are CAT($\kappa$) spaces since these serve as the model space for
applications on manifolds with curvature bounded above.  Ohta \cite{Ohta07}
gives the following relation between CAT($\kappa$)
spaces and $p$-uniformly convex spaces that is a very useful tool for quantifying the
constants involved in characterizing the space.
\begin{lemma}[\cite{Ohta07}, Proposition 3.1]\label{t:CATkappa-pucvx}
  For any CAT$(\kappa)$ space
  $(G, d)$ and any point $\xbar\in G$, for all $\delta\in (0, \pi/(4\sqrt{\kappa}))$ the subspace
  $\paren{\Ball_\delta(\xbar), d}$ is a $p$-uniformly
  convex space with constants $p=2$ and $c_\delta = 4\delta\sqrt{\kappa}\tan\paren{\pi/2 - 2\delta\sqrt{\kappa}}$.
\end{lemma}
\noindent The key point here is that a CAT($\kappa$) space with $\kappa\geq 0$ and small enough diameter is
$p$-uniformly convex with constant $p=2$ and constant $c\in (0,2]$.  In particular, when $\kappa=0$, then
$c=2$.  For convergent fixed point iterations, this means that the regularity of the space to which the iterates
are confined asymptotically approaches a CAT($0$) space, which, despite the absence of addition or scalar multiplication,
shares a lot of structural properties that are enjoyed in linear spaces.

We focus our attention on convex functions $\map{f}{G}{(-\infty, +\infty]}$.  A function is {\em geodesically convex} if it satisfies
\begin{equation*}
\begin{aligned}
&(\forall \tau\in [0,1])(\forall x,y\in G)\\
& f((1-\tau)x\oplus\tau y)\leq (1-\tau)f(x)+\tau f(y).
 \end{aligned}
\end{equation*}
This is equivalent to the epigraph of $f$ being {\em geodesically convex}, i.e., $\epi f$ contains all
geodesics between all pairs of points $(x, \alpha_x)$ and $(y, \alpha_y)$ in the epigraph.
Since geodesic convexity is the only kind of convexity considered here, we will omit the qualifier ``geodesic''.
A function is {\em strongly convex with constant $\mu$} when
\begin{equation*}
\begin{aligned}
&(\forall \tau\in [0,1])(\forall x,y\in G)\\
& f((1-\tau)x\oplus\tau y)\leq (1-\tau)f(x)+\tau f(y) - \tfrac{\mu}{2}(1-\tau)\tau d(x,y)^p.
 \end{aligned}
\end{equation*}

A set-valued mapping $T$ defined on a subset $D\subset G$ to collections of points in $G$ is denoted $\mmap{T}{D}{G}$.
The image of a set $D$ under a mapping $T$ (set-valued or not), written $T(D)$,
is defined by
\begin{equation}\label{eq:image of a set}
 T(D)\equiv \bigcup_{x\in D}T(x).
\end{equation}
The {\em domain} of the mapping $T$ is the set of points in $D$ for which the image is nonempty.  Implicit in
the notation $\mmap{T}{D}{G}$ is the statement that  $T(x)$ is nonempty for all $x\in D$.
 For compositions of set-valued mappings, $\mmap{T_j}{ D_j}{G}$ ($j=1,2, \dots, m$),
 \begin{equation}\label{eq:image of a set composition m}
\begin{aligned}
 T_m\circ \cdots \circ T_2\circ T_1( D_1)&\equiv
 \bigcup_{x_0\in D_1}\paren{\bigcup_{x_1\in T_1(x_0)}\cdots\paren{\bigcup_{x_{m-1}\in T_{m-1}(x_{m-2})} T_m(x_{m-1})}}.
\end{aligned}
\end{equation}
The composition is said to be {\em well-defined} on $D_1$ when it is nowhere empty on $D_1$ and
the images of the ``inner'' mappings are subsets of the domains of definition of the ``outer'' mappings:
$T_{j-1}( D_{j-1})\subset  D_{j}$
for $j\geq 2$.

Key tools for analyzing the behavior of a mapping $\mmap{T}{ D}{G}$ in $p$-uniformly convex spaces
are the following four-point quantities:
For any $x, y, u, v \in G$, we define
\begin{equation}\label{eq:delta}
\begin{aligned}
\Delta^{(p,c)}(x,y,u,v) \equiv  \tfrac{c}{4} \left( d(x, v)^p + d(y, u)^p - d(x, u)^p - d(y, v)^p \right),
\end{aligned}
\end{equation}
and
\begin{equation}\label{eq:psi}
\begin{aligned}
&\psi^{(p,c)}(x,y,u,v) \equiv  \\
&\qquad \tfrac{c}{2} \left( d(x, u)^p + d(y, v)^p + d(u, v)^p + d(x, y)^p - d(y, u)^p - d(x, v)^p \right).
\end{aligned}
\end{equation}
When $u\in T(x)$ and $v\in T(y)$, the mapping $\psi^{(p,c)}$ is referred to as the \emph{transport discrepancy},
introduced first in \cite{BLL}.
We will later show that when $p=c=2$ (e.g., in CAT(0) spaces), $\psi^{(2,2)}(x,y,u,v) \ge 0$.  In a
Hilbert space, the set of distinct quadruplets $(x,y,u,v)$ where $\psi^{(2,2)}(x,y,u,v) = 0$ characterizes the
set of parallelograms.  These objects are key to identifying ``good'' behavior of the mapping $T$ in the space $(G, d)$.
The four-point function defined in \eqref{eq:delta} is a type of parallelogram divergence developed by Nikolaev in \cite{Nikolaev90} and
used in Berg and Nikolaev \cite{Berg08} in the setting of spaces with nonpositive curvature ($CAT(0)$ space).  This was also used
in \cite{Schoetz19} to obtain rates of convergence of the {\em sample barycenter} (i.e. the solution to the problem we are studying)
to the Fr\'echet mean of a random variable over a set (i.e. the limit as the number of points $m$
in the discrete barycenter problem \eqref{eq:frechet mean} goes to infinity).

\subsection{\texorpdfstring{$\alpha$-firmly Nonexpansive Mappings}{alpha-firmly Nonexpansive Mappings}}\label{s:Metric}
Classical contractivity (the Lipschitz continuity with constant $q<1$) is a common assumption that fails
for many, if not most of the successful numerical methods in optimization.  In linear spaces the
assumption that holds generically for mappings derived from {\em convex} functions is {\em nonexpansivity}
(Lipschitz continuity with constant $L=1$)
and {\em firm nonexpansivity} (called {\em firmly contractive} mappings in \cite{Browder67}).
In \cite[Definition 2.2]{LukTamTha18} {\em almost} firmly nonexpansive mappings were proposed to
characterize mappings derived from {\em nonconvex} functions in Euclidean spaces.  In \cite{BLL}
this was extended to $p$-uniformly convex metric spaces and used to develop a convergence theory for
proximal point iterations for convex minimization in Hadamard spaces. The technology for working with
slightly expansive mappings  in \cite{LukTamTha18} was applied to proximal point iterations
for convex minimization in spaces with {\em positive curvature} in \cite{LauLuk21}, where
prox mappings of convex functions are {\em expansive}.

In the next definition, we use the {\em (pointwise) almost $\alpha$-firmly nonexpansive} property,
a more flexible notion that still links each step to the space's geometry through the four-point transport discrepancy $\psi^{(p,c)}$.
\begin{definition}[Definition 5, \cite{LauLuk21}]\label{d:pafne}
 Let $(G, d)$ be a $p$-uniformly convex metric space with constant $c$ and let $D\subset G$.
 \begin{enumerate}[(i)]
  \item The mapping $\mmap{T}{D}{G}$
 is {\em pointwise almost nonexpansive at $y\in D$ on $D$ with violation $\epsilon\geq 0$}
 whenever
\begin{equation}
\label{eq:pane}
\begin{aligned}
\exists \epsilon \geq 0~:~ &\forall x\in D, \forall x_+\in T(x),\text{ and } \forall y_+\in T(y),\\
& d(x_+,y_+)^p\leq (1+\epsilon)d(x,y)^p.
\end{aligned}
\end{equation}
The smallest $\epsilon$
for which \eqref{eq:pane} holds is called the {\em violation}.
If \eqref{eq:pane} holds with $\epsilon=0$, then
$T$ is pointwise nonexpansive at $y\in D$ on $D$.
If \eqref{eq:pane} holds at all $y\in D$ then $T$ is said to be (almost) nonexpansive on $D$.
If $D=G$ the mapping
$T$ is simply said to be (almost) nonexpansive.  If $D\supset \Fix T\neq\emptyset$ and
\eqref{eq:pane} holds at all $y\in \Fix T$ with the same violation, then $T$ is
said to be {\em almost quasi nonexpansive}.
\item  The operator $\mmap{T}{D}{G}$
is {\em pointwise almost $\alpha$-firmly nonexpansive at $y\in D$ on $D$
	with violation $\epsilon\geq 0$}
whenever
\begin{equation}
\label{eq:pafne}
\begin{aligned}
\exists \alpha\in(0,1), \epsilon \geq 0: \forall x\in D_\epsilon(y), \forall x_+\in T(x),\text{ and } \forall y_+\in T(y),\\
\quad  d(x_+,y_+)^p\leq (1+\epsilon)d(x,y)^p-\tfrac{1-\alpha}{\alpha}\psi^{(p,c)}(x,y, x_+,y_+)\quad
\end{aligned}
\end{equation}
If \eqref{eq:pafne} holds with $\epsilon=0$, then
$T$ is pointwise $\alpha$-firmly nonexpansive at $y\in D\subset G$ on $D$.
If \eqref{eq:pafne} holds at all $y\in D$ with the same constant $\alpha$,
then $T$ is said to be (almost) $\alpha$-firmly nonexpansive on $D$.  If $D=G$ the mapping
$T$ is simply said to be (almost) $\alpha$-firmly nonexpansive.  If $D\supset \Fix T\neq\emptyset$ and
\eqref{eq:pafne} holds at all $y\in \Fix T$ with the same constant $\alpha$ then $T$ is
said to be almost quasi $\alpha$-firmly nonexpansive.

\item  The mapping $T$ is said to be {\em pointwise asymptotically  $\alpha$-firmly
	nonexpansive at $y$ with constant $\alpha<1$} whenever
\begin{equation}
\label{eq:asymp-pafne}
\begin{aligned}
\forall \epsilon>0, \exists D_\epsilon(y)\subset G: \forall x\in D_\epsilon(y), \forall x_+\in T(x),\text{ and } \forall y_+\in T(y),
\\
d(x_+,y_+)^p\leq (1+\epsilon)d(x,y)^p-\tfrac{1-\alpha}{\alpha}\psi^{(p,c)}(x,y,x_+,y_+)
\end{aligned}
\end{equation}
  where $D_\epsilon(y)$ is a neighborhood of $y$ in $D$.
 \end{enumerate}
\end{definition}
In what follows, we use the abbreviations \emph{pointwise a$\alpha$--fne} for
pointwise almost $\alpha$-firmly nonexpansive and \emph{pointwise $\alpha$-fne}
for pointwise $\alpha$-firmly nonexpansive mappings. The next propositions provide characterizations of these mappings.
The proofs of these statements are straight-forward and can be found in the extant literature.
\begin{proposition}[pointwise single-valued]
\label{t:single-valued pane}
Let $(G,d)$ be a $p$-uniformly convex
metric space with constant $c > 0$. Any mapping $\mmap{T}{D}{G} (D \subset G)$ that is
pointwise almost nonexpansive at $y\in D$ is single-valued there. Likewise, if $T$
is pointwise almost $\alpha$-fne at $y \in D$, then it is single-valued there.
\end{proposition}
\noindent By the above proposition, whenever a mapping is a$\alpha$-fne at every point in some domain, then it
is single-valued and the notation $T(x)$ is unambiguous.
 \begin{proposition}[characterizations - Proposition 6, \cite{LauLuk21}]
\label{t:properties pafne}
Let  $(G, d)$ be a p-uniformly convex space with constant $c>0$ and
let $\mmap{T}{D}{G}$ for $D\subset G$.
\begin{enumerate}[(i)]
\item\label{t:properties pafne ii}
If $y\in\Fix T$,
\begin{equation}\label{e:psi-Fix T}
\forall x_+\in T(x), \quad \psi^{(p,c)}(x, y,x_+,y)=\tfrac{c}{2}d(x_+, x)^p.
\end{equation}
Also, for fixed $y\in\Fix T$, the function $\psi^{(p,c)}$ satisfies
$$
\psi^{(p,c)}(x,y,x_+,y)\ge 0
\qquad\text{for all } x\in D,
$$
and
$$
\psi^{(p,c)}(x,y,x_+,y)=0
\quad\Longleftrightarrow\quad
x\in\Fix T.
$$\item\label{t:properties pafne iii}  Let $y\in \Fix T$.  $T$ is  pointwise a$\alpha$-fne at $y$
on $D$ with violation $\epsilon\geq 0$ if and only if
\begin{equation}\label{e:P1}
\exists \alpha\in(0,1):\quad 	d(Tx,y)^p\leq (1+\epsilon)d(x,y)^p - \tfrac{1-\alpha}{\alpha}\tfrac{c}{2}d(Tx,x)^p\quad
	\forall x\in D.
\end{equation}
In particular, $T$ is almost quasi $\alpha$-fne on $D$ whenever $T$ possesses fixed points and
\eqref{e:P1} holds at all $y\in \Fix T$ with the same constant $\alpha\in (0,1)$ and violation $\epsilon$.
\item\label{t:properties pafne iv}  If $T$ is pointwise a$\alpha$-fne at $y\in\Fix T$ on $D$ with
constant $\underline\alpha\in(0,1)$ and violation $\epsilon$, then it is pointwise a$\alpha$-fne
at $y$ with the same violation on $D$ for all
$\alpha\in[\underline\alpha,1]$.   In particular, if $T$ is pointwise a$\alpha$-fne at $y\in\Fix T$ on D,
then it is pointwise almost nonexpansive at $y$ with violation $\epsilon$ on D.
\end{enumerate}
\end{proposition}
%

\subsection{Abstract Quantitative Convergence}
Convergence is quantified by a positive strictly monotone function  - called a {\em gauge} - of the fixed point sequence
$(x^k)_{k\in \Nbb}$ relative to $\Fix T$.
Since the gauge of the sequence depends on the
regularity constants of the fixed point mapping $T$, this is written
$\map{\theta_{\alpha,\epsilon}}{[0,\infty)}{[0,\infty)}$,
with parameters $\alpha>0$ and $\epsilon\geq 0$, satisfying
\begin{subequations}\label{eq:theta-Gcal}
\begin{eqnarray}\label{eq:theta_tau_eps}
(i)~ \theta_{\alpha,\epsilon}(0)=0;
\quad (ii)~ 0<\theta_{\alpha,\epsilon}(t)<t ~\forall t\in(0,\tbar]
\mbox{ for some }\tbar>0,
\end{eqnarray}
and for $\alpha\in(0,1)$ fixed and $t\in (0,\overline{t}]$
\begin{equation}\label{eq:gauge}
\begin{gathered}
 \Gcal\paren{\paren{\frac{(1+\epsilon)t^p-\paren{\theta_{\alpha,\epsilon}(t)}^p}{\frac{(1-\alpha)}{\alpha}}}^{1/p}}=t\\
\iff\\
 \theta_{\alpha,\epsilon}(t) = \paren{(1+\epsilon)t^p -
 \tfrac{1-\alpha}{\alpha}\paren{\Gcal^{-1}(t)}^p}^{1/p}.
 \end{gathered}
\end{equation}
\end{subequations}
Here $p\in (1,\infty)$ is the parameter of uniform convexity of the metric space,
and $\Gcal$ is a gauge which is
used to characterize the regularity of the iterate residual $d(x^k, x^{k+1})$.
When $\Gcal$ is simply a linear gauge this becomes
\[
\Gcal(t)=\rho t\quad\iff\quad
\theta_{\alpha, \epsilon}(t)=\paren{(1+\epsilon)-\frac{1-\alpha}{\rho^p\alpha}}^{1/p}t.
\]
The conditions in \eqref{eq:theta_tau_eps} in this
case simplify to $\theta_{\alpha, \epsilon}(t)=\gamma t$, where
\begin{equation}\label{eq:theta linear}
 0< \gamma\equiv \sqrt[p]{1+\epsilon-\frac{1-\alpha}{\rho^p\alpha}}<1\quad\iff\quad
\sqrt[p]{\tfrac{1-\alpha}{(1+\epsilon)\alpha}}\leq  \rho\leq \sqrt[p]{\tfrac{1-\alpha}{\epsilon\alpha}}.
\end{equation}
%
%

\begin{assumption}[regularity assumptions]\label{ass:regularity abstr}
 Let $(G,d)$ be a $p$-uniformly convex metric space with $p\in (1,\infty)$ and constant $c>0$; let $D\subset G$ be
 boundedly compact.  For the multi-valued mapping $\mmap{T}{G}{G}$  and $S:=\Fix T\cap D$ the following assumptions hold.
  \begin{enumerate}[(a)]
  \item \label{ass:regularity 0}  (Self-mapping and Existence)
  $\mmap{T}{D}{D}$ and there is at least one $y \in S$.
  \item\label{ass:regularity 2} (Stability) There is a gauge $\Gcal$ given by \eqref{eq:gauge},
    where the function $\theta_{\alpha,\epsilon}$ satisfies \eqref{eq:theta_tau_eps},
    with
    $t_0\equiv d\paren{x^0, \Fix T\cap D}<\tbar$
    for all $x^0\in D$ and
    \begin{eqnarray}
    d(x, S)    &\le&
    \Gcal\paren{d(x, x_+)}\quad \forall x\in D.\label{e:Psi msr}
  \end{eqnarray}
\end{enumerate}
\end{assumption}
The constants $\alpha$ and $\epsilon$ in \eqref{ass:regularity 2} of Assumption \ref{ass:regularity abstr} have no context
at the moment, but in Proposition \ref{t:msr convergence abstr} below they are exactly the constant and violation of the
a$\alpha$-fne property.  For gauges of a particular form, inequality \eqref{e:Psi msr} is recognizable as an
{\em error bound} \cite{Pang97}.  The inequality \eqref{e:Psi msr} alone is not enough for convergence, but rather it quantifies
how the fixed point sequences corresponding to the mapping $T$ grow or contract.  To obtain
contraction, one of the following assumptions is required.

To set up the next assumption, let $S$ be some nonempty subset of $G$ and
let $t_0\equiv d(x^{0}, S)$.   The sequence $(x^k)_{k\in \Nbb}$ is said to be
 \emph{gauge monotone with respect to $S$ with rate $\Theta$} whenever $\Theta(t)\leq t$
 for all $t\in[0,t_0]$ and
 \begin{equation}\label{e:mu-uniform mon}
(\forall k\in \Nbb) \quad  d(x^{k+1}, S)\leq \Theta\paren{d(x^k, S)}.
 \end{equation}
 The sequence $(x^k)_{k\in \Nbb}$ is said to be
 \emph{linearly monotone with respect to $S$} with rate $c\in [0,1]$ if \eqref{e:mu-uniform mon} is
 satisfied for $\Theta(t)=c\cdot t$ for all $t\in \Rbb_+$.
A  sequence $(x^k)_{k\in \Nbb}$ that converges to a point $\xbar\in S$
is said to converge {\em gauge monotonically} whenever the sequence is
gauge monotone with respect to $S$.  The gauge of monotonicity of the sequence, while related, is not the same
as the rate of convergence of the sequence.  This is clarified below.
\begin{assumption}\label{ass:linear-sublinear conditions}
	Let $(x^k)_{k\in\Nbb}$ be a sequence on the boundedly compact metric space $(G, d)$.
	Let  $S\subset G$ be closed.
	The sequence $(x^k)_{k\in\Nbb}$ is gauge monotone
    with respect to $S$ with rate $\theta_{\alpha,\epsilon}$ satisfying \eqref{eq:theta_tau_eps} where
    $t_0\equiv d(x^0, S)<\tbar$.
    Let $\theta_{\alpha,\epsilon}^{(j)}$ denote the $j$-times composition of
$\theta_{\alpha,\epsilon}$:
$$\theta_{\alpha,\epsilon}^{(j)}(t) = \underset{j \mbox{ times}}{\underbrace{\theta_{\alpha,\epsilon}\circ\theta_{\alpha,\epsilon}\circ\cdots\circ\theta_{\alpha,\epsilon}(t)}}.$$
At least one of the following holds.
  \begin{enumerate}[(i)]
   \item\label{ass:linear-sublinear conditions i}
		The sequence $(x^k)_{k\in\Nbb}$ is Fej\'er monotone with respect to $S$, i.e.
		 \begin{equation}\label{eq:Fejer}
		  d(x^{k+1},\, y)\leq d(x^k, y)\quad \forall\, k\in\Nbb, ~\forall\, y\in S,
		 \end{equation}
          and
          \begin{equation}\label{eq:theta to zero}
            \theta_{\alpha,\epsilon}^{(k)}(t)\to 0\mbox{ as }k\to\infty\quad \forall t\in(0,\tbar).
          \end{equation}
  \item\label{ass:linear-sublinear conditions ii}
		There is a  $\delta>0$ such that
		$d(x^{k+1},\, x^k)\leq \delta d(x^k, S)$ for all $k\in\Nbb$, and
          \begin{equation}\label{eq:theta summable}
            \sum_{j=1}^\infty\theta_{\alpha,\epsilon}^{(j)}(t)<\infty\quad \forall t\in(0,\tbar).
          \end{equation}
  \end{enumerate}
\end{assumption}
Part \eqref{ass:linear-sublinear conditions ii} of Assumption \ref{ass:linear-sublinear conditions} is in some sense stronger than
part \eqref{ass:linear-sublinear conditions i}, but it does not require Fej\'er monotonicity of the sequence.  Typically
\eqref{ass:linear-sublinear conditions i} holds for {\em sublinearly} convergent sequences and \eqref{ass:linear-sublinear conditions ii} holds
for linearly or superlinearly convergent sequences.  Recall that a sequence $(x^k)_{k\in \mathbb{N}}$ is said to
{\em converge $R$-linearly} to some point $\widetilde{x}$ with rate $r\in (0,1)$ whenever
\begin{equation*}
  \exists M<\infty\quad :\, \limsup_{k \to \infty} \frac{d(x^k,\widetilde{x})}{r^k} \leq M.
\end{equation*}

\begin{proposition}[convergence rates, Theorem  16 of \cite{LauLuk21}]\label{t:msr convergence abstr}
 Let $(G,d)$ be a $p$-uniformly convex space with constant $c$; let $D\subset G$ be closed, and let
 $\mmap{T}{G}{G}$ satisfy Assumption \ref{ass:regularity abstr}\eqref{ass:regularity 0} and
 be pointwise a$\alpha$-fne at all $y\in S$ with constant $\alpha$ and violation $\epsilon$ on $D$.
Additionally, let $T$ satisfy Assumption \ref{ass:regularity abstr}\eqref{ass:regularity 2}
with the gauge $\Gcal$ where $\alpha$ and $\epsilon$ are such that $\theta_{\alpha,\epsilon}$
given implicitly by \eqref{eq:gauge} satisfies \eqref{eq:theta_tau_eps}.
Then for any $x^0\in D$, the sequence $(x^k)_{k\in\Nbb}$ defined by
$x^{k+1}\in  T(x^k)$ satisfies
\begin{equation}\label{eq:gauge convergence}
d\paren{x^{k+1},S}
\leq \theta_{\alpha,\epsilon}\paren{d\paren{x^k,S}}
\quad \forall k \in \mathbb{N}.
\end{equation}%
If, in addition,
either one of conditions of Assumption \ref{ass:linear-sublinear conditions}
are satisfied, then the sequence
$(x^k)_{k\in\Nbb}$
converges gauge monotonically  to
some $x^{*}\in S$ with rate given by $O(\theta_{\alpha,\epsilon}^{(k)})$ when
Assumption \ref{ass:linear-sublinear conditions}\eqref{ass:linear-sublinear conditions i}
holds, and with rate $O(s_k(t_0))$ when Assumption \ref{ass:linear-sublinear conditions}\eqref{ass:linear-sublinear conditions ii}
holds, where
$s_k(t)\equiv
\sum_{j=k}^\infty \theta_{\alpha,\epsilon}^{(j)}(t)$ and $t_0\equiv d(x^0,S)$.
%
%
 \end{proposition}

\section{Main Results}
Our focus is on algorithms that are built from compositions and convex combinations of prox mappings.
The $p$-prox mapping of a proper and lower semicontinuous function $f$ in a complete $p$-uniformly convex space is defined by
\begin{align}\label{e:prox-p}
\prox^p_{f, \lambda} (x) \equiv  \argmin_{y \in G} f(y)+ \frac{1}{p \lambda^{p-1}} d(y,x)^p \quad(\lambda>0).
\end{align}
Let $f_j$ ($j=1,\dots,m$) be proper, lsc and convex, and let
$\map{\Tbar}{G}{G}$
 denote one of the following for parameters $\lambda_j>0$ and $\tau_j\in(0,1]$ ($j=1,2,\dots, m$):
\begin{subequations}\label{eq:mapping family}
\begin{align}
\Tbar &\equiv \prox^p_{f_m,\lambda_m} \circ \dots \circ \prox^p_{f_2,\lambda_2} \circ \prox^p_{f_1,\lambda_1};
\label{eq:cyclic_prox} \\
\Tbar &\equiv (\tau_m \prox^p_{f_m,\lambda_m} \oplus (1 - \tau_m)\Id)  \circ \cdots \nonumber \\
  &\quad \circ (\tau_2 \prox^p_{f_2,\lambda_2} \oplus (1 - \tau_2)\Id)
    \circ (\tau_1 \prox^p_{f_1,\lambda_1} \oplus (1 - \tau_1)\Id).
\label{eq:cyclic_gradient_descent} 
\end{align}
\end{subequations}
As a special case of \eqref{eq:cyclic_gradient_descent} with $m=2$, $\tau_1=\tau\in(0,1]$ and $\tau_2=1$, we obtain
$$
\Tbar
= \prox^{p}_{f_2,\lambda_2}\circ\bigl((1-\tau)\Id \oplus \tau\,\prox^{p}_{f_1,\lambda_1}\bigr).
$$
Moreover, \eqref{eq:cyclic_prox} is the specialization of \eqref{eq:cyclic_gradient_descent} with $\tau_j=1$ for all $j=1,\ldots,m$,
so actually, we could treat only \eqref{eq:cyclic_gradient_descent} in our development.  We prefer to keep these separate because,
as we will see, the characterizations of regularity for the extreme case $\tau_i=1$ are different than when $\tau_j<1$.

The case $m=1$ is the classical proximal point algorithm studied in nonpositively curved spaces first in \cite{Jost1995}.
Convergence has been determined more generally in Hadamard spaces by \cite[Theorem 3.4]{Bacak14}
and in complete Alexandrov spaces with curvature bounded above and finite diameter as in Lemma \ref{t:CATkappa-pucvx} by
Ohta and P\'alfia \cite[Theorems 5.1 and 5.5]{OhtaPal15}.
Weak convergence on complete Alexandrov spaces with curvature bounded above  was also obtained in \cite{EspFer09}.

The ultimate goal here is to establish quantitative convergence of iterates of the above mappings, together with an explicit rate.  Beyond
obtaining rates of convergence of iterates, a small but important difference with previous convergence results is the absence of
restrictions on the step sizes $\lambda_j$.  This provides for much simpler implementations.  The cost of this simplicity
is that the fixed points of the algorithms are no longer solutions to \eqref{eq:gen prob}.  But as we show in the numerical demonstration,
the distance of the algorithm fixed point to solutions to \eqref{eq:gen prob} can be estimated and controlled by the step sizes $\lambda_j$.
Depending on the model space, whether Euclidean, CAT($0$) or CAT($\kappa$), the three conditions on the fixed point mapping
in Proposition \ref{t:msr convergence abstr} hold with varying degrees of ease.  We start with the most challenging model space,
spaces with curvature bounded above, and move from there to more restrictive settings.

For prox mappings of convex functions, the assumption in Proposition \ref{t:msr convergence abstr} that $T$ is
a$\alpha$-fne at fixed points
will be refined considerably.  How the regularity of individual prox mappings plays into the regularity of the composite mappings
given by one of \eqref{eq:mapping family} will be developed later and will require the following more highly resolved assumptions.
\begin{assumption}[general regularity assumptions: compositions]\label{ass:regularity gen metric}
Let $D_j\subset G$ for $j=1,2,\dots,m$, and let $\mmap{T_j}{D_j}{G}$.
The composite mapping
$\Tbar\equiv T_m\circ \cdots\circ T_2\circ T_1$ is well-defined on $D_1$ (see \eqref{eq:image of a set composition m}),
and possesses fixed points in $D_1$, and each mapping $T_j$
is a$\alpha$-fne at any $y_j\in T_{j-1}\circ \cdots\circ T_2\circ T_1(\Fix \Tbar\cap D_1)$
with constant $\alpha_j\in (0,1)$
and violation $\epsilon_j\geq 0$ on $D_j\supset T_{j-1}\circ \cdots\circ T_2\circ T_1(D_1)$.
Moreover,
 \begin{equation}
\label{eq:afne composition}
\begin{aligned}
&\exists \alphabar\in (0,1) ~:~ \forall x_0\in D_1, \forall x_{j}\in T_j(x_{j-1})\quad (j=1,2,\dots, m),\\
&\qquad\qquad \frac{1-\alphabar}{\alphabar}\psi^{(p,c)}(x_0,y_0, x_m , y_m)\leq
\sum_{j=1}^m\paren{1+\epsilonhat_{j}} \frac{1-\alpha_{j}}{\alpha_{j}}\psi^{(p,c)}(x_{j-1}, y_{j-1}, x_j , y_j),
\end{aligned}
\end{equation}
where $\epsilonhat_j\equiv \paren{\prod_{i=j+1}^{m}\paren{1+\epsilon_{i}}}-1$
and $\psi^{(p,c)}$ defined by \eqref{eq:psi}.
\end{assumption}

\begin{assumption}[general regularity assumptions: convex combinations]\label{ass:regularity averaged}
For $D\subset G$,  $\mmap{T}{D}{G}$ and $\tau\in [0,1]$, define
$x_\tau \equiv (1-\tau)x\oplus \tau x_+$ for arbitrary $x_+\in T(x)$ and similarly for $y_\tau$ for $x,y\in D$.
Then
\begin{equation}\label{eq:Smile2 prox}
\begin{aligned}
\exists \alpha_\tau\in [0,1)~:\qquad\qquad&\\
\tfrac{1-\alpha_\tau}{\alpha_\tau}\psi^{(p,c)}(x, y, x_\tau, y_\tau)
&\leq \tau^2\tfrac{1-\alpha_c}{\alpha_c}\psi^{(p,c)}(x, y, x_+, y_+)\\
&\qquad  +2(1-\tau)\tau d(x, y)^p \\
  &\qquad  -\, 2(1-\tau)\tau \Delta^{(p,c)}(x, y, x_+, y_+)\\
  &\qquad -\, \tfrac{2-c}{2}(1-\tau)\tau\paren{d(y, x_+)^p+d(x, y_+)^p},
\end{aligned}
\end{equation}
where $\alpha_c$ is given by
\begin{equation}\label{e:alpha_c}
 \alpha_c = \frac{c(c-1)}{2+c(c-1)}.
\end{equation}
\end{assumption}

Assumptions \ref{ass:regularity gen metric} and \ref{ass:regularity averaged} ensure that
pointwise a$\alpha$-fne behavior is preserved under compositions
and convex combinations. Together with the Assumption
\ref{ass:regularity abstr}, this permits the application of Proposition
\ref{t:msr convergence abstr} to the specific mappings in \eqref{eq:mapping family},
yielding the convergence results stated below.

\begin{theorem}[quantitative convergence: cyclic/relaxed proximal point]\label{t:msr convergence}
 Let $(G,d)$ be a $p$-uniformly convex space with $p\in (1,\infty)$ and constant $c\in(3/2,2]$, and let $D\subset G$ be closed.
 Let $\map{f_j}{G}{\Rbb}$ ($j=1,\dots,m$) be proper, lsc and convex and $\lambda_j>0$.

 \begin{enumerate}[(i)]
\item\label{thm: cyclic proximal point i} (prox mappings are a$\alpha$-fne) For all $j$,
$T_j\equiv\prox_{f_j, \lambda_j}^p$ is a$\alpha$-fne on $G$ with constant and violation given by
 \begin{equation}\label{e:alpha-epsilon_c}
	\alpha_{c} \equiv \tfrac{c(c-1)}{c(c-1)+2}\quad\mbox{and violation }\quad \epsilon_c\equiv \tfrac{2-c}{c-1}.
\end{equation}

\item\label{thm: cyclic proximal point ii} (relaxed prox mappings are a$\alpha$-fne). If for all $j$,
$T_{\tau_j, j}\equiv(1-\tau_j)\Id\oplus\tau T_j$ for $T_j\equiv \prox_{f_j, \lambda_j}^p$
satisfies Assumption \ref{ass:regularity averaged} with constant $\alpha_{\tau_j}$, then
$T_{\tau_j, j}$ is a$\alpha$-fne with constant  $\alpha_{\tau_j}$ and violation given by
\begin{equation}\label{e:alpha-epsilon_c 2}
  \epsilon_{\tau_j}\equiv \tau_j^2\epsilon_c
\end{equation}
for $\epsilon_c$ given by \eqref{e:alpha-epsilon_c}.

\item \label{thm: cyclic proximal point iii} ($\Tbar$ is a$\alpha$-fne).
Let $\Tbar$ be given by one of the mappings in \eqref{eq:mapping family} and satisfy
Assumption \ref{ass:regularity  abstr}\eqref{ass:regularity 0} with
$S\equiv \Fix \Tbar\cap D$.
  \begin{enumerate}[(a)]
    \item \label{thm: cyclic proximal point a} For $\Tbar$ given by \eqref{eq:cyclic_prox} suppose that Assumption \ref{ass:regularity gen metric}
    holds on $D=D_1$ with the constant $\alphabar$ for $\epsilon_j=\epsilon_c$ and $\alpha_j=\alpha_c$ in that assumption.
    Define $\epsilonbar=\paren{1+\epsilon_c}^m-1$.

  \item \label{thm: cyclic proximal point iiib} For $\Tbar$ given by \eqref{eq:cyclic_gradient_descent} suppose that
  for all $j$ the mapping $T_{\tau, j}$
  satisfies Assumption \ref{ass:regularity averaged} with constant $\alpha_{\tau_j}$. Let
  Assumption \ref{ass:regularity gen metric} hold on $D=D_1$ with the constant $\alphabar$ for $\epsilon_j=\epsilon_{\tau_j}$
  and $\alpha_j = \alpha_{\tau_j}$ in that assumption.  Define $\epsilonbar = \prod_{j=1}^m (1+\tau_j^2\epsilon_{c})-1$.
  \end{enumerate}
In either case above, $\Tbar$ is a$\alpha$-fne on $D$ with respective constants $\epsilonbar\leq 1$ and $\alphabar\in (0,1)$.
\item \label{thm: cyclic proximal point iv} (general convergence).
In the setting of part \eqref{thm: cyclic proximal point iii}, suppose further that Assumption \ref{ass:regularity abstr}\eqref{ass:regularity 2}
holds for $\Tbar$
with the gauge $\Gcal$ where $\alphabar$ and $\epsilonbar$ are such that $\theta_{\alphabar,\epsilonbar}$
given implicitly by \eqref{eq:gauge} satisfies \eqref{eq:theta_tau_eps}.
Then for any $x^0\in D$, the sequence $(x^k)_{k\in\Nbb}$ defined by
$x^{k+1}\in  \Tbar(x^k)$ satisfies
\begin{equation}\label{eq:gauge convergencebar}
d\paren{x^{k+1},S}
\leq \theta_{\alphabar,\epsilonbar}\paren{d\paren{x^k,S}}
\quad \forall k \in \mathbb{N}.
\end{equation}%
If, in addition,
either one of conditions of Assumption \ref{ass:linear-sublinear conditions}
are satisfied, then the sequence
$(x^k)_{k\in\Nbb}$
converges to
some $x^{*}\in S$ with rate given by $O\paren{\theta_{\alphabar,\epsilonbar}^{(k)}}$ when
Assumption \ref{ass:linear-sublinear conditions}\eqref{ass:linear-sublinear conditions i}
holds, and with rate $O(s_k(t_0))$ when Assumption \ref{ass:linear-sublinear conditions}\eqref{ass:linear-sublinear conditions ii}
holds, where
$s_k(t)\equiv
\sum_{j=k}^\infty \theta_{\alphabar,\epsilonbar}^{(j)}(t)$ and $t_0\equiv d(x^0,S)$.
\item\label{thm: cyclic proximal point linear} (linear convergence). In the setting of part \eqref{thm: cyclic proximal point iii}
suppose that Assumption \ref{ass:regularity abstr}\eqref{ass:regularity 2}
holds for $\Tbar$
with linear gauge $\Gcal(t)\equiv \rho\cdot t$ where the scalar $\rho$ satisfies \eqref{eq:theta linear} for
$\alphabar$ and $\epsilonbar$.
Then, for any $x^0\in D$, any sequence $(x^k)_{k\in\Nbb}$ defined by
$x^{k+1}\in  T(x^k)$ is $R$-linearly convergent to a point in $S$ with
rate constant bounded above by $\gamma=\sqrt[p]{1+\epsilonbar-\frac{1-\alphabar}{\alphabar\rho^p}}$ and
asymptotic rate constant  $\gamma=\sqrt[p]{1-\frac{1-\alphabar }{\alphabar \rho^p}}$.

If $\map{f_j}{G}{\Rbb}$ ($j=1,\dots,m$) is strongly convex with constant $\mu\geq (2-c)/(p\lambda^{p-1})$,
then given any initial point $x_0\in D$, the sequence $x^k\to x^*\in \Fix \Tbar\cap D$  with
rate constant bounded above by $\gamma=\sqrt[p]{1-\frac{1-\alphabar }{\alphabar \rho^p}}$.
\end{enumerate}
 \end{theorem}

For spaces with nonpositive curvature we show
that  neither
Assumption \ref{ass:regularity gen metric} nor Assumption \ref{ass:regularity averaged}
are needed for relaxations of prox mappings of convex functions whenever the relaxation parameter
is small enough, in particular $\tau\in(0,1/2]$.  This leads to the following specialization of
Theorem \ref{t:msr convergence} in Hadamard spaces.
\begin{corollary}[quantitative convergence in Hadamard spaces]\label{t:msr convergence H}
 Let $(G,d)$ be a Hadamard space, let $D\subset G$ be closed, and
 let $\map{f_j}{G}{\Rbb}$ ($j=1,\dots,m$) be proper, lsc and convex.
 Let $\Tbar$ be given by any of the mappings in \eqref{eq:mapping family} where, in case
\eqref{eq:cyclic_gradient_descent} $\tau_j\in(0,1/2]$ for all $j$.
Then the following hold.
 \begin{enumerate}[(i)]
\item\label{cor: hadamard cyclic proximal point i}(relaxed prox mappings are $\alpha$-fne).  $T_{\tau, j}\equiv(1-\tau)\Id\oplus\tau T_j$ for
$T_j\equiv \prox_{f_j, \lambda_j}^p$
is $\alpha$-fne with constant  $\alpha_{\tau, j} = 1-\tau_j$.
\item\label{cor: hadamard cyclic proximal point ii}($\Tbar$ is $\alpha$-fne). If $\Tbar$
satisfies Assumption \ref{ass:regularity abstr}\eqref{ass:regularity 0} on
 $D$, then $\Tbar$ is $\alpha$-fne on $D$ with constant $\alphabar$ satisfying
  \begin{equation}\label{e:alphabar comp cat0}
\sup_{x_0\in D}\frac{d(x_0, x_m)^2}%
  {d(x_0, x_m)^2+ \sum_{j=1}^m\frac{1-\alpha_{j}}{\alpha_{j}}\psi^{(2,2)}(x_{j-1}, y_{j-1}, x_j , y_j)}
  \leq \alphabar <1.
 \end{equation}
 \item\label{cor: hadamard cyclic proximal point iii gen} (general convergence).
Let Assumption \ref{ass:regularity abstr} hold in its entirety
for $\Tbar$ with the gauge $\Gcal$, where  $\epsilonbar=0$ and $\alphabar$ is such that $\theta_{\alphabar,0}$
given implicitly by \eqref{eq:gauge} satisfies \eqref{eq:theta_tau_eps}.
Then for any $x^0\in D$, the fixed point sequence $(x^k)_{k\in\Nbb}$ satisfies \eqref{eq:gauge convergencebar}.
If, in addition,
either one of conditions of Assumption \ref{ass:linear-sublinear conditions}
are satisfied, then the sequence
$(x^k)_{k\in\Nbb}$
converges to
some $x^{*}\in S$ with rate given by $O\paren{\theta_{\alphabar,0}^{(k)}}$ when
Assumption \ref{ass:linear-sublinear conditions}\eqref{ass:linear-sublinear conditions i}
holds, and with rate $O(s_k(t_0))$ when Assumption \ref{ass:linear-sublinear conditions}\eqref{ass:linear-sublinear conditions ii}
holds, where
$s_k(t)\equiv
\sum_{j=k}^\infty \theta_{\alphabar,0}^{(j)}(t)$ and $t_0\equiv d(x^0,S)$.

 \item\label{cor: hadamard cyclic proximal point iii} (linear convergence).
 In the setting of part \eqref{cor: hadamard cyclic proximal point iii gen},
let the gauge of stability be  linear, $\Gcal(t)\equiv \rho\cdot t$, where the scalar $\rho$ satisfies \eqref{eq:theta linear} for
$\epsilonbar=0$ and the given $\alphabar$.
Then, for any $x^0\in D$, any sequence $(x^k)_{k\in\Nbb}$ defined by
$x^{k+1}\in  T(x^k)$ is $R$-linearly convergent to a point in $\Fix \Tbar\cap D$ with
rate constant bounded above by $\gamma=\sqrt[p]{1-\frac{1-\alphabar}{\alphabar\rho^p}}$.
\end{enumerate}
 \end{corollary}

 To conclude this section, note that the main results rely on verifying that each mapping in \eqref{eq:mapping family}
 is pointwise a$\alpha$-fne (with the appropriate constants and violations). In Section \ref{section:proofs}
 these properties are established in detail, together with additional statements specific to nonpositively curved spaces.

 \section{Proofs of the Main Results}\label{section:proofs}
To prove Theorem \ref{t:msr convergence} and the corollary we first develop the calculus of regularity of composite and
relaxed mappings in general $p$-uniformly convex spaces.  This is subsequently specialized to spaces with nonpositive curvature
and to prox mappings in both general spaces and Hadamard spaces.  This section concludes with the proofs of the main results.
 \subsection{\texorpdfstring{Calculus of a$\alpha$-fne Operators}{Calculus of a-alpha-fne Operators}}\label{s:calculus}
 Not included in the assumptions of Theorem \ref{t:msr convergence} is the pointwise a$\alpha$-fne property of
 the composite fixed point mapping. For this we require a calculus for compositions and convex combinations of pointwise a$\alpha$-fne mappings.
\begin{proposition}[compositions, Lemma 9 of \cite{LauLuk21}]
 \label{t:compositionthm}
 Let $(G, d)$ be a $p$-uniformly convex space with constant $c>0$ and let $D\subset G$.
 Let  $\mmap{T_1}{D}{G}$ be pointwise a$\alpha$-fne at $y$ on
 $D$ with constant $\alpha_{1}$ and violation $\epsilon_{1}$
 and let $\mmap{T_2}{T_1(D)}{G}$ be pointwise a$\alpha$-fne at $y$ on
 $T_1(D)$ with constant
 $\alpha_{2}$ and violation $\epsilon_{2}$.

 If there exists a constant $\alphabar\in (0,1)$ such that
 \begin{equation}\label{eq:afne composition2}
 \begin{aligned}
&\frac{1-\alphabar}{\alphabar}\psi^{(p,c)}(x,y, x_+, y_+)\\
&\quad \leq (1+\epsilon_1)
\frac{1-\alpha_1}{\alpha_1}\psi^{(p,c)}(x,y, x_1, y_1)
 + \frac{1-\alpha_2}{\alpha_2}\psi^{(p,c)}(x_1,y_1, x_2, y_2)
\quad \forall x\in D,
\end{aligned}
\end{equation}
then the composite operator $\Tbar={T_2} \circ T_1$ is pointwise a$\alpha$-fne at
 $y$ on $D$ with violation  $\epsilonbar=\epsilon_{1}+\epsilon_{2} + \epsilon_{1}\epsilon_{2}$ with constant $\alphabar$.
\end{proposition}

Assumption \ref{ass:regularity gen metric} generalizes \eqref{eq:afne composition2} and will be used in the following result.

\begin{corollary}[finite compositions of pointwise a$\alpha$-fne mappings]\label{c:compositionthm}
 Suppose $(G, d)$ be a $p$-uniformly convex space with constant $c>0$ and $D\subset G$.
 Let  $\mmap{T_1}{D_1}{G}$ be pointwise a$\alpha$-fne at $y$ on
 $D$ with constant $\alpha_{1}$ and violation $\epsilon_{1}$
 and let $\mmap{T_2}{T_1(D)}{G}$ be pointwise a$\alpha$-fne at $y$ on
 $T_1(D)$ with constant
 $\alpha_{2}$ and violation $\epsilon_{2}$.

 $T_1 : D_1 \rightarrow G$ where $D_1 \subset G$ and for $j = 2,3,\ldots,m$ let $T_j : D_j \rightarrow G$ for
$D_j \equiv  \{T_{j-1}x \,|\, x \in D_{j-1}\}$.
If the mappings $T_j$ are pointwise a$\alpha$-fne with constant $\alpha_j$ on $D_j (j = 1,2,\ldots, m)$ and
Assumption \ref{ass:regularity gen metric} holds,
then the composite operator $\Tbar \equiv  T_m \circ \cdots \circ T_1$ is pointwise a$\alpha$-fne on $D_1$
with constants $\alphabar \in (0,1)$ and
$\epsilonbar = \paren{\prod_{i=1}^{m}\paren{1+\epsilon_{i}}}-1$.
\end{corollary}
\begin{proof}
    The result follows by the previous proposition and an induction argument.
\end{proof}
\begin{lemma}[convex relaxations]\label{l:cvx relaxation id}
  Let $(G,d)$ be a $p$-uniformly convex space with $p\in(1,\infty)$ and constant $c > 0$, and let $\mmap{T}{D}{G}$ for
  $D\subset G$. Define $T_\tau\equiv (1-\tau)\Id\oplus \tau T$ for $\tau \in [0,1]$.
  For $x_+\in T(x)$ and $y_+\in T(y)$ denote
  $x_\tau \equiv (1-\tau)x\oplus \tau x_+$ and $y_\tau \equiv (1-\tau)y\oplus \tau y_+$.  Then
  \begin{equation}\label{e:Big Thief}
    \begin{aligned}
    d(x_\tau,y_\tau)^p
    &\le (1-\tau)^2 d(x,y)^p + \tau^2 d(x_+,y_+)^p + 2(1-\tau)\tau\,\Delta^{(p,c)}(x,y,x_+,y_+) \\
    &\quad + \tfrac{2-c}{2}(1-\tau)\tau\big(d(y,x_+)^p + d(x,y_+)^p\big).
    \end{aligned}
  \end{equation}
\end{lemma}
\begin{proof}
  Since $G$ is a $p$-uniformly convex space, for the geodesic from $x$ to $x_+$,
  $$
  d(x_\tau,y_\tau)^p \le (1-\tau)\,d(x,y_\tau)^p + \tau\,d(x_+,y_\tau)^p - \tfrac{c}{2}(1-\tau)\tau\,d(x,x_+)^p.
  $$
  Applying the same inequality to the geodesic from $y$ to $y_+$ gives \eqref{e:Big Thief}.
\end{proof}

\begin{theorem}[convex relaxations at generic points] \label{t:cvx relaxation}
  Let $(G,d)$ be a $p$-uniformly convex space with $p\in(1,\infty)$ and constant $c > 0$, and let $\mmap{T}{D}{G}$ for
$D\subset G$ and for all $x\in D$ let $T_\tau(x)= (1-\tau)x\oplus \tau T(x)$ for some $\mmap{T}{D}{G}$
and $\tau \in [0,1]$. For $x_+\in T(x)$ denote $x_\tau \equiv (1-\tau)x\oplus \tau x_+$.
$T$ is pointwise almost nonexpansive at $y$ with violation $\epsilon$ on $D$ if and only
if, for all $x\in D$, for all $x_+ \in T(x)$,
\begin{align*}
  d(x_\tau ,T_\tau(y))^p\\
  \leq (1 &+ \epsilon_\tau) d(x,y)^p -(1-\tau)\tau\psi^{(p,c)}(x,y,x_+,T(y))\\
&+ (1-\tau )\tau \frac{2-c}{2} (d(y,x_+)^p + d(x,T(y))^p),
\end{align*}
where
\begin{equation}\label{e:Sweet}
  \epsilon_\tau = \tau\!\left(2\tau + (1-\tau)c + \epsilon\Big(\tau + \tfrac{c}{2}(1-\tau)\Big) - 2\right),
\end{equation}
\end{theorem}
\begin{proof}
  Using almost nonexpansiveness at $y$ gives $d(x_+,y_+)^p \le (1+\epsilon)d(x,y)^p$.
  Substituting this inequality into \eqref{e:Big Thief} results in
  \begin{align*}
  d(x_\tau,y_\tau)^p
  &\le \Big[(1-\tau)^2 + (1-\tau)\tau\tfrac{c}{2} + \big(\tau^2 + (1-\tau)\tau\tfrac{c}{2}\big)(1+\epsilon)\Big] d(x,y)^p \\
  &\quad - (1-\tau)\tau\,\psi^{(p,c)}(x,y,x_+,y_+) + \tfrac{2-c}{2}(1-\tau)\tau\big(d(y,x_+)^p + d(x,y_+)^p\big).
  \end{align*}
  Simplifying the coefficient of $d(x,y)^p$ gives
  $$
  (1-\tau)^2 + \tau^2 + c(1-\tau)\tau + \epsilon\Big(\tau^2 + \tfrac{c}{2}(1-\tau)\tau\Big)
  = 1 + \epsilon_\tau,
  $$
  where
  $$
  \epsilon_\tau = \tau\!\left(2\tau + (1-\tau)c + \epsilon\Big(\tau + \tfrac{c}{2}(1-\tau)\Big) - 2\right),
  $$
  which matches \eqref{e:Sweet}. Hence,
  $$
  \begin{aligned}
  d(x_\tau ,T_\tau(y))^p
  &\le (1+\epsilon_\tau)\, d(x,y)^p - (1-\tau)\tau\,\psi^{(p,c)}(x,y,x_+,T(y)) \\
  &\quad + \tfrac{2-c}{2}(1-\tau)\tau\big(d(y,x_+)^p + d(x,T(y))^p\big).
  \end{aligned}
  $$
  Conversely, assume that the inequality of the theorem holds for all $\tau\in[0,1]$.
  Evaluating it at $\tau=1$ proves that $T$ is pointwise almost nonexpansive at $y$.
\end{proof}

\begin{proposition}[convex relaxations - sufficient conditions]\label{t:cvx relaxation suff}
  Let  $(G, d)$ be a p-uniformly convex space with constant $c>0$, and
let $\mmap{T}{D}{G}$ for $D\subset G$.  For all $x\in D$ let
$T_\tau(x) = (1-\tau) x\oplus\tau T (x)$ for some $\mmap{T}{D}{G}$
and $\tau\in [0,1]$.  For $x_+\in T(x)$ denote $x_\tau\equiv (1-\tau)x\oplus \tau x_+$.
\begin{enumerate}[(i)]
  \item\label{t:cvx relaxation i}
  Let $T$ be pointwise almost nonexpansive at $y$ with violation $\epsilon$ on $D$. Then
$T_\tau$ is pointwise a$\alpha$-fne at $y$ with constant $\alpha = 1-\tau \in (0,1]$ and violation
$\epsilon_\tau$ given by \eqref{e:Sweet} if, for all $x\in D$, for all $x_+\in T(x)$,
\begin{equation}\label{eq:Bitter2}
  \begin{aligned}
  \tfrac{\tau}{1-\tau}\psi^{(p,c)}(x, y, x_\tau, T_\tau(y))
  &\leq (1-\tau)\tau\psi^{(p,c)}(x, y, x_+, T(y))\\
    &\qquad -\, \tfrac{2-c}{2}(1-\tau)\tau\paren{d(y, x_+)^p+d(x, T(y))^p}.
  \end{aligned}
  \end{equation}
  \item\label{t:cvx relaxation Ttilde suff}
Let $T$ be pointwise a$\alpha$-fne at $y$ with constant $\alpha\in(0,1)$ and
violation $\epsilon$ on $D$.  Then $T_\tau$ is pointwise a$\alpha$-fne at $y$ with constant
$\alpha_\tau\in(0,1)$ and violation $\epsilon_\tau=\tau^2\epsilon$ if, for all $x\in D$, for all $x_+\in T(x)$,

\begin{equation}\label{eq:Smile2}
\begin{aligned}
\tfrac{1-\alpha_\tau}{\alpha_\tau}\psi^{(p,c)}(x, y, x_\tau, T_\tau(y))
&\leq \tau^2\tfrac{1-\alpha}{\alpha}\psi^{(p,c)}(x, y, x_+, T(y))\\
&\qquad  +\paren{2(1-\tau)\tau}d(x, y)^p \\
  &\qquad -\, (1-\tau)\tau 2\Delta^{(p,c)}(x, y, x_+, T(y))\\
  &\qquad -\, \tfrac{2-c}{2}(1-\tau)\tau\paren{d(y, x_+)^p+d(x, T(y))^p}.
\end{aligned}
\end{equation}
\end{enumerate}
\end{proposition}

\begin{proof}
  For part \eqref{t:cvx relaxation i}, assume $T$ is pointwise almost nonexpansive at $y$
  with violation $\epsilon$ on $D$. By Theorem \ref{t:cvx relaxation}, for any $x\in D$ and any $x_+\in T(x)$,
  $$
  \begin{aligned}
  d(x_\tau ,T_\tau(y))^p
  &\le (1+\epsilon_\tau)\, d(x,y)^p -(1-\tau)\tau\,\psi^{(p,c)}(x,y,x_+,T(y)) \\
  &\qquad + \tfrac{2-c}{2}(1-\tau)\tau\big(d(y,x_+)^p + d(x,T(y))^p\big),
  \end{aligned}
  $$
  with $\epsilon_\tau$ given by \eqref{e:Sweet}. Using \eqref{eq:Bitter2} gives
  $$
  d(x_\tau ,T_\tau(y))^p \le (1+\epsilon_\tau)\, d(x,y)^p - \tfrac{\tau}{1-\tau}\,\psi^{(p,c)}(x,y,x_\tau,T_\tau(y)).
  $$
  Therefore $T_\tau$ is
  pointwise $\alpha$-fne at $y$ with $\alpha=1-\tau$ and violation $\epsilon_\tau$, as claimed.

  Part \eqref{t:cvx relaxation Ttilde suff}.  Starting with  \eqref{e:Big Thief}
   for all $x \in D$ and $x_+ \in T(x)$, we have
  \begin{equation}\label{eq:cvx relaxation anfe}
    \begin{aligned}
      d(x_\tau,y_\tau)^p
      &\le (1-\tau)^2 d(x,y)^p
         + \tau^2d(x_+, y_+)^p + 2(1-\tau)\tau\,\Delta^{(p,c)}(x,y,x_+,T(y)) \\
      &\quad + \frac{2-c}{2}(1-\tau)\tau
         \big(d(y,x_+)^p + d(x,T(y))^p\big)\\
      &\le \paren{(1-\tau)^2  + \tau^2(1+\epsilon)}\,d(x,y)^p
         - \tau^2 \frac{1-\alpha}{\alpha}\,
           \psi^{(p,c)}(x,y,x_+,T(y)) \\
      &\quad + 2(1-\tau)\tau\,\Delta^{(p,c)}(x,y,x_+,T(y)) \\
      &\quad + \frac{2-c}{2}(1-\tau)\tau
         \big(d(y,x_+)^p + d(x,T(y))^p\big),
      \end{aligned}
      \end{equation}
where the second inequality holds because $T$ is pointwise a$\alpha$-fne at $y$ with violation $\epsilon$ and constant $\alpha$.
Substituting \eqref{eq:Smile2} into \eqref{eq:cvx relaxation anfe} yields
  \begin{equation}\label{eq:afne-Ttau}
  \begin{aligned}
  & d(x_\tau,y_\tau)^p
   \\
  &\quad\le  \paren{(1-\tau)^2 + \tau^2(1+\epsilon) + 2(1-\tau)\tau}\,d(x,y)^p
   - \frac{1-\alpha_\tau}{\alpha_\tau}\,
     \psi^{(p,c)}(x,y,x_\tau,y_\tau)\\
    &\quad =\paren{1+\tau^2\epsilon}\,d(x,y)^p
   - \frac{1-\alpha_\tau}{\alpha_\tau}\,
     \psi^{(p,c)}(x,y,x_\tau,y_\tau),
     \end{aligned}
  \end{equation}
  which proves $T_\tau$ is pointwise
  $\alpha$-fne at $y$
  with violation $\epsilon_\tau = \tau^2\epsilon$, as claimed.
  \end{proof}

Notice in Proposition \ref{t:cvx relaxation suff} that the violation $\epsilon_\tau$
of the relaxed mapping $T_\tau$ can be reduced by taking smaller steps $\tau$.  This is
exactly in line with the intuition for reducing step sizes:  the more irregular the mapping $T$,
the more cautious one should be in taking the steps suggested by this mapping.

\subsection{Spaces with Nonpositive Curvature}\label{s:calculus CAT(0)}

In order to specialize the general framework of Section \ref{s:calculus} to the setting of
nonpositively curved metric spaces, recall from Lemma \ref{t:CATkappa-pucvx} that a CAT($0$) space
 is $p$-uniformly convex with $p=c=2$.  We therefore restrict our attention
to $p$-uniformly convex spaces with $p=c=2$.

\begin{lemma}\label{t:metric space properties}
  Let $(G,d)$ be a p-uniformly convex metric space with constant $p=c=2$. Then
  \begin{align} \label{e:metric CS}
    \Delta^{(2,2)}(x,y,u,v)\le d(x,y)d(u,v) \quad \forall x,y,u,v\in G,
  \end{align}
  and
  \begin{align} \label{e:metric CS psi}
    \psi^{(2,2)}(x,y,u,v)\ge 0 \quad \forall x,y,u,v\in G.
  \end{align}
\end{lemma}
\begin{proof}
  For the proof of inequality \eqref{e:metric CS} see \cite[Lemma 2.1]{LanPavSch00}.
  For \eqref{e:metric CS psi}, it is easy to verify for all $p$ and $c$,
  \begin{equation}\label{eq:delta_to_psi}
    \psi^{(p,c)}(x,y,u,v) = \frac{c}{2}(d(x,y)^p + d(u,v)^p) -2\Delta^{(p,c)}(x,y,u,v).
  \end{equation}
  Now using \eqref{e:metric CS} yields
  \begin{equation*}
    \begin{aligned}
      \psi^{(2,2)}(x,y,u,v) &\ge d(x,y)^2 + d(u,v)^2 - 2d(x,y)d(u,v)\\
      &= (d(x,y) - d(u,v))^2 \ge 0.
    \end{aligned}
  \end{equation*}
\end{proof}

\begin{lemma}\label{l:compsite condition flat}
Let $(G, d)$ be a $p$-uniformly convex metric space with $p=c=2$,
for $D_j\subset G$ ($j=1,2,\dots,m$) let $\mmap{T_j}{D_j}{G}$ be such that
the composite mapping
$\Tbar\equiv T_m\circ \cdots\circ T_2\circ T_1$ is well-defined on $D_1$,
and possesses fixed points in $D_1$, and each mapping $T_j$
is a$\alpha$-fne at any $y_j\in T_{j-1}\circ \cdots\circ T_2\circ T_1(\Fix \Tbar\cap D_1)$
with constant $\alpha_j\in (0,1)$
and violation $\epsilon_j\geq 0$ on $D_j\supset T_{j-1}\circ \cdots\circ T_2\circ T_1(D_1)$.
Then there exists an $\alphabar\in (0,1)$ such that \eqref{eq:afne composition} holds.
\end{lemma}
  \begin{proof}
  The proof is by contradiction.  First note that if \eqref{eq:afne composition} holds for
  some $\alphabar\in(0,1)$, then it holds for all $\alpha\geq\alphabar$.  So if there does not exist such
  an $\alphabar$, then for some $x_0\in D_1$ and  $x_{j}\in T_j(x_{j-1}), (j=1,2,\dots, m)$
\begin{equation}\label{e:Fred again}
\psi^{(p,c)}(x_{0}, y_{0}, x_m , y_m)>0 \und 0= \sum_{j=1}^m\paren{1+\epsilonhat_{j}}
\frac{1-\alpha_{j}}{\alpha_{j}}\psi^{(p,c)}(x_{j-1}, y_{j-1}, x_j , y_j).
\end{equation}
By construction $y_m\in T_{m}\circ \cdots\circ T_2\circ T_1(y_0)$ for $y_0\in\Fix \Tbar\cap D_1$, so $y_0=y_m$ and
$$\tfrac{c}{2} d(x_0, x_m)^p = \psi^{(p,c)}(x_{0}, y_{0}, x_m , y_m)>0.$$
Here we have used the fact that, because the mappings $T_j$  are pointwise a$\alpha$-fne at $y_{j}$ they
are single-valued there (Proposition \ref{t:single-valued pane}), so the notation
$y_{j+1}=T_j(y_{j})=\Tbar_j(y_1)$ is unambiguous and $y_0=y_m=T_m(y_{m-1})$.

Also note that $\psi^{(p,c)}$ is nonnegative by Lemma \ref{t:metric space properties},
so the sum in \eqref{e:Fred again} can only be zero when the transport discrepancy $\psi^{(p,c)}$ is zero for
each $j=1,2,\dots,m$.  This is the metric characterization of parallel transport of the pairs $(x_j, y_j)$ for
each $j=0,1,\dots, m$.  Since the point $y_{m-1}$ is transported to $y_m$ by the mapping $T_{m}$, and the
point $x_{m-1}$ is transported to the point $x_m$ ``parallel'' to $y_m=y_0$ by the same mapping, this implies that
$x_m=x_0$, in contradiction to $ d(x_0, x_m)^p >0$.
  \end{proof}

\begin{proposition}[compositions at fixed points - nonpositive curvature]
\label{t:nonexp cat(0)}
Let $(G, d)$ be a $p$-uniformly convex space with $p=c=2$, and let
$D_j\subset G$ be nonempty for $j=1,2, \dots,m$.   Let the composition
$\Tbar_j\equiv T_{j}\circ\cdots\circ T_2\circ T_1$  where $\mmap{T_j}{D_j}{G}$
for $j=1, 2\dots, m$ be well-defined for all $x\in D_1$.
For $j=1,2, \dots, m$, let $x_0, y_0\in D_1$, and let $T_j$
be pointwise a$\alpha$-fne  at $y_{j-1}\in D_{j}$ with
constant $\alpha_j$ and  violation $\epsilon_j$ on $D_{j}$,
 where $\{y_{j}\}= T_{j}(y_{j-1}) = \Tbar_j(y_0)$ and the violations $\epsilon_j$ satisfy
 \begin{equation} \label{e:general composite violation}
  \epsilonbar_m\equiv \paren{\prod_{j=1}^{m}\paren{1+\epsilon_{j}}}-1\leq 1.
 \end{equation}
 Denote $\epsilonhat_j\equiv \paren{\prod_{i=j+1}^{m}\paren{1+\epsilon_{i}}}-1$ and $x_j\in T_j(x_{j-1})$
 ($j=1,2,\dots,m$).  Then when $\Tbar_m$ possesses fixed points, it is pointwise a$\alpha$-fne
 at $y_0\in \Fix \Tbar_m$ with violation $\epsilonbar_m$ given by \eqref{e:general composite violation} on $D_1$
and constant $\alphabar$ satisfying \eqref{e:alphabar comp cat0}.
\end{proposition}

\begin{proof}
First note that since the mappings $T_j$  are pointwise a$\alpha$-fne at $y_{j}$, they
are single-valued there (Proposition \ref{t:single-valued pane}), so the notation
$y_{j}=T_j(y_{j-1})=\Tbar_j(y_0)$ is unambiguous.

By Lemma \ref{l:compsite condition flat} Assumption \ref{ass:regularity gen metric} holds, so by
Corollary \ref{c:compositionthm} the composite mapping $\Tbar_m$ is
pointwise a$\alpha$-fne at $y_0$.  Rearranging the terms in
\eqref{eq:afne composition} yields
\eqref{e:alphabar comp cat0}.
\end{proof}

\begin{remark}
  In Proposition \ref{t:nonexp cat(0)}, if each mapping $T_j$ is $\alpha$-fne
  (with $\epsilon_j=0$ for all $j$), then by \eqref{e:general composite violation}
  the composite mapping $\Tbar_m$ also has zero violation and is $\alpha$-fne.
\end{remark}

\begin{theorem}[convex combinations - nonpositive curvature]
\label{t:cvx comb flat}
Let $D\subset G$ where $(G, d)$ is a $p$-uniformly convex space with constant $p=c=2$ and
let $T_1, T_2: D\rightrightarrows G$.  Denote $T_\tau\equiv (1-\tau) T_2\oplus \tau T_1$ for
$\tau\in [0,1]$.
\begin{enumerate}[(i)]
\item\label{t:cvx comb flat i} Any convex combination of $T_2$ and $T_1$ is
pointwise almost nonexpansive at $y\in D$ with violation
$\overline\epsilon\equiv\max_{j}{\epsilon_j}$ on $D$ whenever $T_2$ and $T_1$
are pointwise almost nonexpansive there
with respective violations $\epsilon_1$, and $\epsilon_2$.

\item\label{t:cvx relaxation flat suff ane}
Let $T_2=\Id$ and let $T_1=T$ be pointwise almost nonexpansive at $y$ with
violation $\epsilon$ on $D$.
Then $T_\tau$ is pointwise a$\alpha$-fne at $y$
with constant
$\alpha=1-\tau$ and violation $\epsilon_\tau=\epsilon\tau$
whenever $\tau\in(0,1/2]$.

 \item\label{t:cvx relaxation Ttau - flat} Let $T_2=\Id$ and denote $T_1=T$.
 The mapping $T_\tau$ is pointwise a$\alpha$-fne at $y$ with constant $\alpha_\tau$ and
violation $\epsilon_\tau\in (0,1)$ on $D$ if and only if
\begin{equation}\label{eq:proto-mon Ttilde - flat}
\begin{aligned}
&T  \, \mbox{ is single-valued at $y$, and }\forall x\in D,  \forall x_+\in T(x)\\
& \paren{\frac{1}{2\tau}+\frac{\alpha_\tau}{2(1-\alpha_\tau)\tau}}d\paren{x_\tau, T_\tau(y)}^2\\
&\qquad \leq \paren{\tfrac{(1+\epsilon_\tau)\alpha_\tau}{2\tau(1-\alpha_\tau)}
+ \tfrac{(1-\tau)}{\tau} - \tfrac{1}{2\tau}}d(x, y)^2
 + \Delta^{(2,2)}\paren{x,y,x_+, T (y)},
\end{aligned}
\end{equation}
where $\alpha_\tau, \epsilon_\tau\in (0,1)$.

\item\label{t:cvx relaxation flat suff}
Let $T_2=\Id$ and let $T_1=T$ be pointwise a$\alpha$-fne at $y$ with constant
$\alpha\in(0,1)$ and violation $\epsilon$ on $D$.
Then $T_\tau$ is pointwise a$\alpha$-fne at $y$ with constant
$\alpha_\tau\in(0,1)$ and violation $\epsilon_\tau=\tau^2\epsilon$ if,
for all $x\in D$, for all $x_+\in T(x)$,
\begin{equation}\label{eq:Smile3}
\begin{aligned}
\tfrac{1-\alpha_\tau}{\alpha_\tau}\psi^{(2,2)}(x, y, x_\tau, T_\tau(y))
&\leq \tau^2\tfrac{1-\alpha}{\alpha}\psi^{(2,2)}(x, y, x_+, T(y))\\
&\qquad  +2(1-\tau)\tau d(x, y)^2 \\
  &\qquad -\, 2(1-\tau)\tau \Delta^{(2,2)}(x, y, x_+, T(y)).
\end{aligned}
\end{equation}
\end{enumerate}

\end{theorem}

\begin{proof}
Part \eqref{t:cvx comb flat i}.
Since $T_1$ and $T_2$ are pointwise almost nonexpansive at $y$, they are single-valued there (Proposition \ref{t:single-valued pane})
so $T_j(y)$ ($j=1,2$) is unambiguous.
Let $\tau\in(0,1)$ and define $T_\tau \equiv (1-\tau)T_2\oplus\tau T_1$.
Applying \eqref{e:Big Thief} with $p=c=2$ yields for all $x\in D$ and all
$x_1\in T_1(x)$, $x_2\in T_2(x)$,
\begin{align*}
d(x_\tau,T_\tau (y))^2&\leq
(1-\tau)^2d(x_2,T_2(y))^2+\tau^2 d(x_1,T_1(y))^2\\
&\qquad (1-\tau)\tau 2\Delta^{(2,2)}(x_2, T_2(y), x_1, T_1(y)),
\end{align*}
where $\Delta^{(2,2)}$ is defined by \eqref{eq:delta}.
On the other hand $(G,d)$ is $p$-uniformly convex with $p=c=2$, so \eqref{e:metric CS} holds, and in particular,
for any points in $G$
$$ 2\Delta^{(2,2)}(x_2, T_2(y), x_1, T_1(y))\leq  2d(x_2,T_2(y))d(x_1,T_1(y)).$$
Therefore the right hand side of the inequality above can be written as a perfect square
$$
d(x_\tau,T_\tau (y))^2\leq \paren{(1-\tau)d(x_2,T_2(y))+\tau d(x_1,T_1(y))}^2
\quad\forall x\in D.
$$
Since $T_2$ and $T_1$ are pointwise almost nonexpansive at $y$ with violations
$\epsilon_j\leq \overline\epsilon\equiv\max_{j}{\epsilon_j}$ on $D$, so
for all $x\in D$
$$
d(x_\tau,T_\tau (y))^2\leq ((1-\tau)\sqrt{1+\epsilonbar}d(x,y)+\tau \sqrt{1+\epsilonbar}d(x,y))^2
=(1+\epsilonbar)d(x,y)^2.
$$
and hence $d(x_\tau,T_\tau (y))\leq  \sqrt{1+\epsilonbar}d(x,y)$ for all $x\in D$ as claimed.

Part \eqref{t:cvx relaxation flat suff ane}.  By Proposition
\ref{t:cvx relaxation suff}, $T_\tau$ is a$\alpha$-fne with
$\alpha=1-\tau\in(0,1]$ and violation $\epsilon_\tau$ given by \eqref{e:Sweet} if,
for all $x\in D$, for all $x_+\in T(x)$, \eqref{eq:Bitter2} holds. For the case with $p=c=2$,
\eqref{e:Sweet} is just $\epsilon_\tau=\epsilon\tau$, and
\eqref{eq:Bitter2} is
 \begin{equation}\label{eq:Bitter2 0}
\begin{aligned}
\tfrac{\tau}{1-\tau}\psi^{(2,2)}(x, y, x_\tau, T_\tau(y))&\leq
(1-\tau)\tau\psi^{(2,2)}(x,y,x_+, T(y)).
\end{aligned}
\end{equation}
By expanding $x_\tau$ and $T_\tau(y)$,
\begin{equation}\label{eq:warm gun 0}
\begin{aligned}
&\psi^{(2,2)}(x,y,x_\tau, T_\tau(y))\leq
 \tau^2\psi^{(2,2)}(x,y, x_+,T(y)),
\end{aligned}
\end{equation}
so \eqref{eq:Bitter2 0} holds when
 \begin{equation*}
\begin{aligned}
\tfrac{\tau}{1-\tau}\tau^2 \psi^{(2,2)}(x,y, x_+,T(y))
&\leq (1-\tau)\tau \psi^{(2,2)}(x,y, x_+,T(y)),
\end{aligned}
\end{equation*}
which holds when
$$
 \tfrac{\tau}{1-\tau}\tau^2\leq (1-\tau)\tau,
$$
or, in other words, when $\tau\in(0,1/2]$ as claimed.

Parts \eqref{t:cvx relaxation Ttau - flat} and \eqref{t:cvx relaxation flat suff} are just
the respective specializations of Theorem \ref{t:cvx relaxation} and Proposition
\ref{t:cvx relaxation suff}  with $p=c=2$.
\end{proof}

\subsection{Prox Mappings}
We recall the $p$-prox mapping  introduced in \eqref{e:prox-p}.  This is nonempty and single-valued for a proper, lsc, and convex
function $f$  \cite[Proposition 2.7]{Izuchukwu19}.
This operator plays a central role in splitting algorithms for optimization problems.
In linear spaces the prox mapping  has been extensively studied.  Prox mappings of convex functions in linear spaces
are  globally nonexpansive and
$\alpha$-fne.  In nonlinear spaces this no longer holds, but it does hold {\em almost} \cite{BLL} and \cite{LauLuk21}.  A few key facts
about this mapping are collected next.

\begin{lemma}[prox inequality]\label{l:prox_char}
  Let $(G, d)$ be a boundedly compact, $p$-uniformly convex metric space with $p \in (1, \infty)$ and $c > 0$. Let $f : G \to \mathbb{R}$ be proper, lsc, and convex, and fix $\lambda \in (0, \infty)$. Then $x_+ = \prox^p_{f, \lambda}(x)$ if and only if
  \begin{equation}\label{eq:prox_char}
  \frac{4}{cp\lambda^{p-1}} \Delta^{(p,c)}(x_+, x, x_+, y) - \frac{2-c}{2p \lambda^{p-1}} d(x_+, y)^p \leq f(y) - f(x_+) \quad \forall y \in G.
  \end{equation}
  \end{lemma}

  \begin{proof}
  To begin, $x_+ = \prox_{f, \lambda}(x)$ if and only if for all $z \in G$
  $$
  f(x_+) + \tfrac{1}{p \lambda^{p-1}} d(x, x_+)^p \leq f(z) + \tfrac{1}{p \lambda^{p-1}} d(z, x)^p.
  $$
In particular, for $z_\tau := (1 - \tau)y + \tau x_+$, with $y \in G$, rearranging the proximal inequality and using
  convexity of $f$ together with the uniform convexity of the space with constant $c$ yields
  $$
  \begin{aligned}
  f(x_+) &\leq f(z_\tau) + \tfrac{1}{p \lambda^{p-1}} d(x, z_\tau)^p - \tfrac{1}{p \lambda^{p-1}} d(x, x_+)^p \\
  &\leq (1 - \tau)f(y) + \tau f(x_+) + \tfrac{1}{p \lambda^{p-1}} d(z_\tau, x)^p - \tfrac{1}{p \lambda^{p-1}} d(x, x_+)^p \\
  &\leq (1 - \tau)f(y) + \tau f(x_+) + \tfrac{1-\tau}{p \lambda^{p-1}} d(y, x)^p + \tfrac{\tau}{p \lambda^{p-1}} d(x, x_+)^p \\
  &\quad - \tfrac{c\tau(1 - \tau)}{2p\lambda^{p-1}} d(x_+, y)^p - \tfrac{1}{p\lambda^{p-1}} d(x, x_+)^p.
  \end{aligned}
  $$
  Equivalently, gathering terms gives
  $$
  \frac{1 - \tau}{p\lambda^{p-1}} \left( d(x_+, x)^p - d(x, y)^p + \tfrac{c}{2}\tau d(x_+,y)^p \right) \leq (1 - \tau)(f(y) - f(x_+)).
  $$
  Dividing both sides by $(1 - \tau)$ and letting $\tau \nearrow 1$ yields
  $$
  \frac{1}{p\lambda^{p-1}} \left( \frac{c}{2}d(x_+, y)^p + d(x_+, x)^p - d(y, x)^p \right) \leq f(y) - f(x_+).
  $$
  Using $c/2 = 1 - (2 - c)/2$ and the definition of $\Delta^{(p,c)}$ from \eqref{eq:delta} completes the proof.

  For the converse, note that \eqref{eq:prox_char} is the worst case for $\tau \in (0, 1)$, so it is sufficient for preceding inequalities, hence the reverse implication also holds.
  \end{proof}

\begin{theorem}[prox mappings of convex functions: general metric spaces]\label{t:prox props gen}
  Let $(G,d)$ be a boundedly compact, $p$-uniformly convex metric space with constant $p>1$ and $c>0$
  and let $f\colon G \rightarrow \Rbb$ be proper, convex and lsc. Then the following hold.
  \begin{enumerate}[(i)]
    \item\label{t:prox props gen 0} $\arg\min f \subset \Fix \prox_{f,\lambda}^p.$
    \item\label{t:prox props gen 1} For all $x,y \in G$,
    \begin{equation} \label{metrically monotone}
      \Delta^{(p,c)}(x, y, x_+, y_+) \geq \frac{c^2}{4} d(x_+,y_+)^p
    \end{equation}
    where $x_+=prox_{f,\lambda}^p(x)$, and similarly for $y_+$.
    \item\label{t:prox props gen 2}
    Let the parameter $c$ be in the semi-closed interval $(3/2,2]$. Then the prox mapping is a$\alpha$-fne on $G$,
    that is, for $x_+= prox^p_{f,\lambda}(x)$ and $y_+= prox^p_{f,\lambda}(y)$,
     \begin{equation*}
    \begin{aligned}
    d(x_+,y_+)^p \leq (1 + \epsilon_c) d(x,y)^p - \frac{1-\alpha_c}{\alpha_c} \psi^{(p,c)}(x,y,x_+,y_+)\quad \forall x,y \in G,
    \end{aligned}
    \end{equation*}
    with constant $\alpha_c$ and violation $\epsilon_c$ given by
     \begin{equation}\label{e:alpha-epsilon_c 0}
    \begin{aligned}
    \alpha_c = \frac{c(c-1)}{2+c(c-1)}, \quad \epsilon_c =\frac{2-c}{c-1}.
    \end{aligned}
    \end{equation}
    \end{enumerate}
\end{theorem}
\begin{proof}
  To prove \eqref{t:prox props gen 0}, let $\xbar\in \argmin f$. Then
  \begin{align*}
    f (\xbar) = f(\xbar) + \frac{1}{p\lambda^{p-1}} d(\xbar,\xbar)^p
    \leq f(y) + \frac{1}{p\lambda^{p-1}} d(y,\xbar)^p \quad \forall y\in G.
  \end{align*}
  So $\xbar=\prox^p_{f,\lambda}(\xbar)$.

Part \eqref{t:prox props gen 1}.
For $x\in G$ and $x_+=\prox_{f,\lambda}^p(x)$, by Lemma \ref{l:prox_char}, the inequality \eqref{eq:prox_char} holds for all $y\in G$,
and in particular for $y_+=\prox_{f,\lambda}^p(y)$ because
$$
\frac{4}{cp\lambda^{p-1}} \Delta^{(p,c)}(x_+, x, x_+, y_+) - \frac{2-c}{2p \lambda^{p-1}} d(x_+, y_+)^p \leq f(y_+) - f(x_+) \quad \forall y \in G.
$$
Similarly, exchanging the roles of $x$ and $y$,
$$
\frac{4}{cp\lambda^{p-1}} \Delta^{(p,c)}(y_+, y, y_+, x_+) - \frac{2-c}{2p \lambda^{p-1}} d(x_+, y_+)^p \leq f(x_+) - f(y_+) \quad \forall x \in G.
$$
Adding the two inequalities yields
$$
c\, d(x_+, y_+)^p \leq d(x, y_+)^p + d(y, x_+)^p - d(x, x_+)^p - d(y, y_+)^p = \frac{4}{c} \Delta^{(p,c)}(x, y, x_+, y_+),
$$
which proves \eqref{t:prox props gen 1}.

For part \eqref{t:prox props gen 2}, writing inequality \eqref{metrically monotone} in terms of the
transport discrepancy in \eqref{eq:delta_to_psi} gives
\begin{align*}
\frac{c^2}{4} d(x_+,y_+)^p &\leq \frac{c}{4} (d(x,y)^p + d(x_+,y_+)^p) - \frac{1}{2}\psi^{(p,c)}(x, y, x_+, y_+),\\
d(x_+,y_+)^p &\leq \frac{1}{c-1} d(x,y)^p - \frac{2}{c(c-1)}\psi^{(p,c)}(x, y, x_+, y_+).\\
\end{align*}
Let $1+\epsilon = 1/(c-1)$ and $(1-\alpha) = 2/(c(c-1))$. Then the prox-mapping is a$\alpha$-fne with constants in \eqref{e:alpha-epsilon_c 0}.
In order to ensure that $\epsilon<1$, it suffices to require $c\in(3/2,2]$.

  \end{proof}
\begin{proposition}[proximal distance to a set]\label{ex:prox_sqdist_set}
Let $(G, p)$ be $p$-uniformly convex metric space with
$p\in (1,\infty)$ and $c>0$, and let $\Omega\subset G$ be closed.  The prox mapping of the Moreau-Yosida envelope of the indicator
of the set $\Omega$ has the representation
\begin{equation}\label{e:prox of the distance}
\begin{aligned}
&\prox^p_{\tfrac{1}{p}d(\cdot, \Omega)^p, \lambda}(x) = \set{\tfrac{1}{1+\lambda}x\oplus \tfrac{\lambda}{1+\lambda} y}{y\in P_\Omega(x)}.
\end{aligned}
\end{equation}
\end{proposition}
\begin{proof}
 Take any $y\in P_\Omega(x)$, which is nonempty because $\Omega$ is closed.  Then
 \[
  \prox^p_{\tfrac{1}{p}d(\cdot, y)^p, \lambda}(x)\subset \prox^p_{\tfrac{1}{p}d(\cdot, \Omega)^p, \lambda}(x).
 \]
Note that $\prox^p_{\tfrac{1}{p}d(\cdot, y)^p, \lambda}(x)$ is the barycenter of the pair of points
$x$ and $y$ with weight $\omegahat=\lambda^{p-1}/(1+\lambda^{p-1})$.  Indeed, letting $\omega=\lambda/(1+\lambda)$,
 \begin{equation*}
 \begin{aligned}
  \prox^p_{\tfrac{1}{p}d(\cdot, y)^p, \lambda}(x)&\equiv\argmin_z\klam{\tfrac{1}{p}d(z, y)^p+\tfrac{1}{p\lambda^{p-1}}d(z,x)^p}\\
  &=\argmin_z\klam{\omega^{p-1}d(z, y)^p+(1-\omega)^{p-1}d(z,x)^p}\\
  &=\argmin_z\klam{\omegahat d(z, y)^p+(1-\omegahat)d(z,x)^p},
 \end{aligned}
 \end{equation*}
 where $\omegahat = \omega^{p-1}/((1-\omega)^{p-1} + \omega^{p-1})$.
The barycenter exists and is unique \cite[Proposition 3.3]{Kuwae2014}, and  by
\cite[Proposition 3.6]{Kuwae2014} has the representation
$z_*(y) = (1-\tauhat_*)x\oplus \tauhat_* y$ where $\tauhat_*$ is
\[
\tauhat_* =  \frac{\omegahat^{\tfrac{1}{p-1}}}{(1-\omegahat)^{\tfrac{1}{p-1}} + \omegahat^{\tfrac{1}{p-1}}} = \omega = \tfrac{\lambda}{1+\lambda}.
\]

The remaining claim of the theorem is that $\prox^p_{\tfrac{1}{p}d(\cdot, \Omega)^p, \lambda}(x)$ is the union over
points $z_*(y)$ derived in this way from $y\in P_\Omega(x)$.  To see this choose any point $y\notin P_\Omega(x)$
and observe that for all $\lambda>0$ large enough,
 \[
  \prox^p_{\tfrac{1}{p}d(\cdot, y)^p, \lambda}(x)\notin \prox^p_{\tfrac{1}{p}d(\cdot, \Omega)^p, \lambda}(x).
 \]
\end{proof}

It is an easy exercise to show that for $f = d(\cdot, y)^p$ the prox mapping
$\prox_{f,\lambda}^p(x)$ is $\alpha$-fne on $G$ for all prox-parameters
$\lambda\geq \paren{(2-c)/(pc)}^{1/(p-1)}$.  More generally, we have the following result.
\begin{proposition}[prox mappings of strongly convex functions are $\alpha$-fne]\label{prop:prox-alpha-fne}
Let $(G, d)$ be a boundedly compact, $p$-uniformly convex metric space with $p \in (1, \infty)$ and $c > 0$,
and let $f : G \to \mathbb{R}$ be proper, lsc, and strongly convex with parameter $\mu \geq (2 - c)/(p \lambda^{p-1})$.

  \begin{enumerate}[(i)]
\item \label{prop: strogly cvx i} Let $x_+ = \prox_{f, \lambda}^p(x)$. Then
\begin{equation} \label{prop: i}
  \frac{4}{cp\lambda^{p-1}} \Delta^{(p,c)}(x_+, x, x_+, y) \leq f(y) - f(x_+) \quad \forall y \in G.
\end{equation}

\item \label{prop: strogly cvx ii} The prox mapping is $\alpha$-fne on $G$, that is, for $x_+ = \prox_{f, \lambda}^p(x)$ and $y_+ = \prox_{f, \lambda}^p(y)$,
$$
d(x_+, y_+)^p \leq d(x, y)^p - \frac{1 - \alpha}{\alpha} \psi^{(p,c)}(x, y, x_+, y_+) \quad \forall x, y \in G,
$$
with constant $\alpha = c/(2 + c)$.
\end{enumerate}

\end{proposition}

\begin{proof}

    For \eqref{prop: strogly cvx i}, the process of the proof is almost identical to Lemma \ref{l:prox_char}. If $x_+ = \prox_{f, \lambda}(x)$,
    $$
    f(x_+) + \tfrac{1}{p \lambda^{p-1}} d(x, x_+)^p \leq f(z_\tau) + \tfrac{1}{p \lambda^{p-1}} d(z_\tau, x)^p,
    $$
  for $z_\tau \equiv (1-\tau)y\oplus \tau x_+$. Then by rearranging the inequality and using strong convexity of $f$
  together with uniform convexity of the space yields
$$
  \begin{aligned}
    f(x_+) &\leq f(z_\tau) + \tfrac{1}{p \lambda^{p-1}} d(x, z_\tau)^p - \tfrac{1}{p \lambda^{p-1}} d(x, x_+)^p \\
    &\leq (1 - \tau)f(y) + \tau f(x_+) - \tfrac{\mu \tau(1-\tau)}{2} d(y, x_+)^p +\tfrac{1}{p \lambda^{p-1}} d(z_\tau, x)^p - \tfrac{1}{p \lambda^{p-1}} d(x, x_+)^p \\
    &\leq (1 - \tau)f(y) + \tau f(x_+) - \tfrac{\mu \tau(1-\tau)}{2} d(y, x_+)^p + \tfrac{1-\tau}{p \lambda^{p-1}} d(y, x)^p + \tfrac{\tau}{p \lambda^{p-1}} d(x, x_+)^p \\
    &\quad - \tfrac{c\tau(1 - \tau)}{2p\lambda^{p-1}} d(x_+, y)^p - \tfrac{1}{p\lambda^{p-1}} d(x, x_+)^p.
  \end{aligned}
 $$
 Dividing both sides by $(1-\tau)$ and taking the limit $\tau \nearrow 1$ gives
 $$
 \tfrac{\mu}{2}d(x_+,y)^p + \frac{1}{p\lambda^{p-1}} \left( \tfrac{c}{2}d(x_+, y)^p + d(x_+, x)^p - d(y, x)^p \right) \leq f(y) - f(x_+).
 $$
 As in Lemma \ref{l:prox_char}, using $c/2 = 1 - (2 - c)/2$ and the definition of $\Delta^{(p,c)}$ from \eqref{eq:delta}, yields

 $$
 \paren{\frac{\mu}{2}-\frac{2-c}{2p\lambda^{p-1}}}d(x_+,y)^p +\frac{4}{cp\lambda^{p-1}} \Delta^{(p,c)}(x_+, x, x_+, y) \leq f(y) - f(x_+) \quad \forall y \in G.
$$
By the assumption $\mu \ge (2-c)/ (p\lambda^{p-1})$, the coefficient multiplying $d(x_+,y)^p$ is nonegative and the claim is proved.

For \eqref{prop: strogly cvx ii}, applying \eqref{prop: i} to $x_+ = \prox^p_{f,\lambda}(x)$ on $y_+$ and
similarly to $y_+ = \prox^p_{f,\lambda}(y)$ and adding them up, yields
  $$
   \Delta^{(p,c)}(x_+, x, x_+, y) + \Delta^{(p,c)}(y_+, y, y_+, x) \leq 0.
  $$
 This can be written
  \begin{align*}
  &d(x_+, y_+)^p \le \\
  &\quad d(x, y)^p - \paren{d(x, x_+)^p + d(y, y_+)^p + d(x_+, y_+)^p + d(x, y)^p - d(y, x_+)^p - d(x, y_+)^p}.
  \end{align*}
  With definition of $\psi^{(p,c)}(x,y,x_+,y_+)$ in \eqref{eq:psi} and $\alpha = c/(2+c)$, the proof is complete.
  \end{proof}

\subsection{Prox Mappings: Spaces with Nonpositive Curvature}\label{s:prox_cat0}

Throughout this subsection we specialize the results above to the case where $(G,d)$ is a $p$-uniformly
convex metric spaces with constants $p=c=2$, i.e. a  CAT($0$) space.

\begin{proposition}[prox mappings of convex functions: nonpositive curvature]\label{t:prox_cat0}
  Let $(G,d)$ be a boundedly compact, $p$-uniformly convex metric space with constants $p=c=2$,
  and let $f\colon G \to \Rbb$ be proper, convex and lsc. Then $\prox^2_{f,\lambda}$
  is everywhere $\alpha$-fne with $\alpha = 1/2$ and everywhere nonexpansive.
  Moreover, $\argmin f = \Fix \prox^2_{f,\lambda}$.
\end{proposition}
\begin{proof}
Substituting $p=c=2$ in \eqref{e:alpha-epsilon_c 0} gives $\alpha = 1/2$ and $\epsilon = 0$, which proves
$\prox^2_{f,\lambda}$ is $\alpha$-fne with $\alpha=1/2$ (no violation). That $\prox^2_{f,\lambda}$
is nonexpansive follows from \eqref{e:metric CS psi} of Lemma \ref{t:metric space properties}, which yields
$\psi^{(2,2)}(x,y, x_+, y_+)\ge 0$ for all $x, y, x_+, y_+$.
To see that $\argmin f= \Fix \prox^2_{f,\lambda}$, the inequality \eqref{eq:prox_char} of Lemma \ref{l:prox_char}
for $c=2$ and $x_+\in \Fix \prox^2_{f,\lambda}$ yields $0\le f(y)-f(x_+)$ for all $y\in G$, hence $x_+\in\argmin f$.
Together with Theorem \ref{t:prox props gen}\eqref{t:prox props gen 0}, the proof is complete.
\end{proof}
The next corollary is an immediate specialization to indicator functions of convex sets.
\begin{corollary}[projectors in Hadamard spaces]\label{t:proj_hadamard}
Let $(G,d)$ be a Hadamard space and let $C\subset G$ be nonempty, closed and convex.
For any $x\in G$, the projector $P_C(x)\equiv  \argmin_{z\in C} d(x,z)$ is nonempty, a singleton, nonexpansive
and $\alpha$-fne with constant $\alpha=1/2$.
\end{corollary}

\begin{corollary}[convex relaxations of prox mappings of convex functions]\label{t:cvx comb prox flat}
In the setting of Theorem \ref{t:prox_cat0}, whenever $\tau\in(0,1/2]$,
the relaxed prox mapping  $T_\tau\equiv (1-\tau)\Id\oplus\tau\prox^2_{f,\lambda}$ is $\alpha$-fne
with constant $\alpha_\tau = 1-\tau$.
\end{corollary}
\begin{proof}
 The proof follows from Proposition \ref{t:prox_cat0} and
 Theorem \ref{t:cvx comb flat}\eqref{t:cvx relaxation flat suff ane}.
\end{proof}

\subsection{Proof of the Main Result}

\begin{proof}[Proof of Theorem \ref{t:msr convergence}]
  Part \eqref{thm: cyclic proximal point i}. By Theorem \ref{t:prox props gen}\eqref{t:prox props gen 2}, each
  $T_j:=\prox_{f_j,\lambda_j}^p$ is pointwise a$\alpha$-fne on $G$ with
  $ \alpha_c=c(c-1)/(c(c-1)+2),$ and
  $\epsilon_c=(2-c)(c-1)$,
  which is well-defined for $c\in(3/2,2]$.\\
  Part \eqref{thm: cyclic proximal point ii}.
 Since each
  $T_j:=\prox_{f_j,\lambda_j}^p$ is pointwise a$\alpha$-fne on $G$ with
  $\alpha_c$ and $\epsilon_c$, using Assumption \ref{ass:regularity averaged} and Proposition \ref{t:cvx relaxation suff} yields each $T_{\tau,j}$ is pointwise a$\alpha$-fne at $y_{j-1}$ with $\alpha_{\tau,j}$ and $\epsilon_{\tau,j} = \tau^2\epsilon_c$.\\
  Part \eqref{thm: cyclic proximal point iii}.
  For the first part,
  Assumption \ref{ass:regularity abstr}\eqref{ass:regularity 0} guarantees $\Tbar(D)\subseteq D$ and
  $S:=\Fix\Tbar\cap D\neq\emptyset$.
  By Assumption \ref{ass:regularity gen metric} and Corollary \ref{c:compositionthm},
  $\Tbar$ is pointwise a$\alpha$-fne at every $y\in \Fix\Tbar\cap D$ with a common constant
  $\alphabar\in(0,1)$ and violation
  $
  \epsilonbar \;=\; \prod_{j=1}^m(1+\epsilon_c)-1.
  $
  The same argument applies to the second part, with $\epsilon_{\tau,j}$ replacing $\epsilon_c$ in the violation term.\\

  Parts \eqref{thm: cyclic proximal point iv} and \eqref{thm: cyclic proximal point linear}.
  Assumption \ref{ass:regularity abstr}\eqref{ass:regularity 2} supplies the stability
  $d(x,S)\le \rho\, d(x,\Tbar x)$ with $\rho$ satisfying \eqref{eq:theta linear}, with $\alphabar,\epsilonbar$.
  Therefore all assumptions of Proposition \ref{t:msr convergence abstr} hold for $\Tbar$ and the result follows from this for the general
  case, and the specialization of this via \eqref{eq:theta linear} to the case of linear gauges.
  To see the claimed asymptotic rate of convergence in the linear case, note that the iterates are confined
  to vanishingly small neighborhoods of the fixed point $\xbar$. By Lemma \ref{t:CATkappa-pucvx}
  the constant $c_\delta$ on these
  vanishingly small neighborhoods $\Ball_\delta(\xbar)$ converges to $2$ from below.  The violation
  of the $\alpha$-fne property on $\Ball_\delta(\xbar)$ given by \eqref{e:alpha-epsilon_c} therefore converges to
  zero as $\delta\searrow 0$.  The asymptotic rate
  is therefore $\gamma \le \sqrt[p]{1- \frac{1-\alphabar }{\alphabar \rho^p}}$.\\
  If we further assume $f_i$'s are strongly convex with $\mu\ge (2-c)/(p\tau_j^{p-1})$, by
  Proposition \ref{prop:prox-alpha-fne}, the prox mappings are $\alpha$-fne on $G$.
  This implies $\epsilon_j=0$ for all $j$, so the composite violation satisfies $\epsilonbar  =0$ and the
  asymptotic rate above applies to the entire sequence for any $x_0\in D$.
  \end{proof}

\begin{proof}[Proof of Corollary \ref{t:msr convergence H}]
  Part \eqref{cor: hadamard cyclic proximal point i}.
  Since $(G,d)$ is a Hadamard space, it is a complete CAT$(0)$ space, hence $p$-uniformly convex with $p=c=2$.
  For each $j$, let $ T_j \equiv  \prox^2_{f_j,\lambda_j}$.
  Then Corollary \ref{t:cvx comb prox flat} yields that $T_{\tau,j}$ is pointwise $\alpha$-fne with $\alpha_{\tau,j}=1-\tau_j$.

  Part \eqref{cor: hadamard cyclic proximal point ii}.
  In case \eqref{eq:cyclic_prox}, each factor $\prox^2_{f_j,\lambda_j}$ is pointwise $\alpha$-fne with $\alpha=\frac12$
  and zero violation by Proposition \ref{t:prox_cat0}.
  In case \eqref{eq:cyclic_gradient_descent}, part \eqref{cor: hadamard cyclic proximal point i} shows that each factor $T_{\tau,j}$ is pointwise $\alpha$-fne with $\alpha_{\tau,j}=1-\tau_j$.
  Hence, in either case, all factors in the composition defining $\Tbar$ are pointwise $\alpha$-fne with zero violation.
  Since $\Fix \Tbar \cap D$ is nonempty by Assumption \ref{ass:regularity abstr}\eqref{ass:regularity 0},
  Proposition \ref{t:nonexp cat(0)} implies that
  $\Tbar$ is pointwise $\alpha$-fne at every $y\in \Fix \Tbar \cap D$ for $\alphabar$ satisfying \eqref{e:alphabar comp cat0}
  (existence guaranteed by Lemma \ref{l:compsite condition flat}).

  Parts \eqref{cor: hadamard cyclic proximal point iii gen} and \eqref{cor: hadamard cyclic proximal point iii}.
  By part \eqref{cor: hadamard cyclic proximal point ii}, the mapping $\Tbar$ is pointwise $\alpha$-fne on $D$.
  Together with Assumption \ref{ass:regularity abstr}, all hypotheses of Proposition \ref{t:msr convergence abstr} are satisfied.
  Therefore, every sequence generated by $\Tbar$ converges $R$-linearly to a point in $\Fix \Tbar \cap D$ with rate determined by the gauge $\Gcal$.
  \end{proof}

  \section{Applications}\label{s:numerics}

  The elementary problem we consider is the computation of the Fréchet mean introduced in \eqref{eq:frechet mean}.
  In the context of \eqref{eq:gen prob} each function $f_j$ is just the distance to a point $x_j$, so the
  prox mapping is just $\prox_{f_j, \lambda_j}(x) = 1/(1+\lambda_j)x\oplus 1/(1+\lambda_j)x_j$
  (see Proposition \ref{ex:prox_sqdist_set}).  The mappings
  in \eqref{eq:mapping family} are therefore all the same. We use the same relaxation parameters $\lambda_j$
  for each data point.
  The fixed points of these mappings are {\em not}
  solutions to \eqref{eq:frechet mean}, but they are quantifiable approximations. Indeed, the solution to
  the Fréchet mean problem lies in the convex hull of the data points. By the same reasoning, the
  solution to the Fréchet mean problem also lies in the convex hull of the intermediate points of the mappings in
  \eqref{eq:cyclic_gradient_descent}. Therefore, the distance from fixed points of $\Tbar$ to solutions of
  \eqref{eq:frechet mean} can be bounded above by the diameter of the intermediate points of the mappings in
  \eqref{eq:cyclic_gradient_descent}. This diameter goes to zero as the stepsizes $\tau_j\searrow 0$ for all $j$.
  The diameter is computed as the largest pairwise distance between the intermediate points.

  In a $p$-uniformly convex space, the Fr\'echet mean problem is modeled by choosing
  $
  f_j(x)\equiv \omega_j d(x,x_j)^p,
  $
  for data points $\{x_1,\ldots,x_m\}$ and weights $\omega_j\in(0,1)$ satisfying $\sum_{j=1}^m \omega_j=1$.
  Each $f_j$ is proper, lower semicontinuous, and convex, so its prox mapping is well defined.
  Moreover, Proposition \ref{ex:prox_sqdist_set} provides a closed-form expression
  for the corresponding $\prox^p$ mapping in a $p$-uniformly convex space.
  This allows us to construct the mapping $\Tbar$ in \eqref{eq:cyclic_gradient_descent},
  and hence the fixed point iteration described in Algorithm \ref{algo:cyclic_gd}.

  \begin{algorithm}[H]
    \caption{Cyclic gradient descent}
    \label{algo:cyclic_gd}
    \begin{algorithmic}[1]
        \STATE{{\bf Initialization.} Functions $f_1,\dots,f_m$ on $G$; positive parameters $\{\lambda_i>0\}_{i=1}^m$;
        parameter $\tau\in[0,1]$; initial $x_0\in G$; tolerance $\varepsilon>0$.}
        \STATE{{\bf General step ($k=0,1,2,\dots$):  While $d(x_{k+1},x_k)\le \varepsilon$ do }\\
            $x_{k+1} = T(x_k) \equiv
            (\tau \operatorname{prox}^2_{f_m,\lambda_m} \oplus (1 - \tau)\mathrm{Id})\circ \cdots
            \circ (\tau \operatorname{prox}^2_{f_1,\lambda_1} \oplus (1 - \tau)\mathrm{Id})(x_k)$
        }
    \end{algorithmic}
\end{algorithm}

%

  The theory developed in previous sections guarantees $R$-linear convergence of the iterates generated by
  Algorithm \ref{algo:cyclic_gd} to fixed points of $\Tbar$ under the stated regularity assumptions.
  The numerical results illustrate this behavior on two model spaces: the manifold of symmetric positive definite matrices with
  the affine invariant metric, and the sphere with the usual Riemannian metric.

  \subsection{SPD matrices}
The space of symmetric positive definite matrices with the affine invariant metric is denoted by $(\mathcal{S}_{++}^d, d)$
where
$
\mathcal{S}_{++}^d = \{ A \in \mathbb{R}^{d \times d} \,:\, A^\top = A,\;
x^\top A x > 0 \ \text{for all } x \in \mathbb{R}^d \setminus \{0\} \},
$
and
$
d(A,B) = \Big\| \log\big(A^{-1/2} B A^{-1/2}\big) \Big\|_F,$
where $A,B \in \mathcal{S}_{++}^d$,
for $\|\cdot\|_F$ the Frobenius norm. The metric is a Riemannian metric, and this space is a Hadamard manifold
\cite[Theorem 10.39]{BirHae99}. By Lemma \ref{t:CATkappa-pucvx}
it is also $p$-uniformly convex metric space with parameters $p=c=2$.
Therefore Corollary \ref{t:msr convergence H} applies in this setting.  Moreover, Assumption \ref{ass:regularity abstr}
can be removed entirely since every collection of points in every complete, $p$-uniformly convex metric space possesses
a unique barycenter \cite{Kuwae2014},
and the Fr\'echet function \eqref{eq:frechet mean} satisfies the K\L $~$ property
\cite[Proposition 3.5]{CruzNetoOliveira19}.  The proof that this implies condition \eqref{e:Psi msr} of
Assumption \ref{ass:regularity abstr}\eqref{ass:regularity 2} is left to future work, but suffice it to say that
the key stepping stone is the equivalence between metric regularity of the epigraphical mapping associated with the
Fr\'echet function and the latter satisfying the  K\L $~$ inequality \cite[Corollary 4]{BolDanLeyMaz10} which implies
gauge metric {\em subregularity} of the mapping $\Phi(x)=\set{d(x,x_+)}{x_+\in T(x)}$ for $0$ at $y\in \Fix T$, which is
condition \eqref{e:Psi msr} of Assumption \ref{ass:regularity abstr}\eqref{ass:regularity 2}.

For the numerical experiments, sample points in $\mathcal{S}_{++}^d$ were generated by a
construction using the Cholesky decomposition. We sampled random lower triangular matrices $L$ with positive
diagonal entries and set $A=LL^\top$. This yields symmetric positive definite matrices by
construction.

Under Assumptions \ref{ass:regularity abstr}, Corollary \ref{t:msr convergence H} predicts that
the iterates generated by Algorithm \ref{algo:cyclic_gd} converge globally $R$-linearly to a
fixed point of the cyclic gradient descent mapping for relaxation parameter $\tau\in(0,1/2]$.  For the
Fr\'echet mean objective, the proximal parameter and the relaxation parameter are indistinguishable.
By Proposition \ref{t:prox_cat0} the prox of any convex function in a CAT($0$) space is $\alpha$-fne
with constant $\alpha=1/2$, there is in fact no restriction on the relaxation parameter $\tau$ for which
Corollary \ref{t:msr convergence H} holds.  Figure \ref{fig:applications_convergence}(left) demonstrates this
predicted linear behavior for several values of the relaxation parameter $\tau$.
Smaller values of $\tau$ lead to smaller cycle diameters, and hence to more accurate
approximations of the Fréchet mean. 
For $\tau=0.9$  the maximum pairwise distance between points in the final cycle of Algorithm \ref{algo:cyclic_gd}
was $3.34763$;  for $\tau=0.005$  the maximum pairwise distance between points in the final cycle
was $0.0197798$.  The ratio of $\tau$ to the cycle diameter was roughly constant over all the different values of $\tau$
(around $0.25$) indicating a linear dependence of the cycle diameter with the relaxation parameter.

\begin{figure}[t]
  \centering
    \centering
    \includegraphics[width=0.48\linewidth]{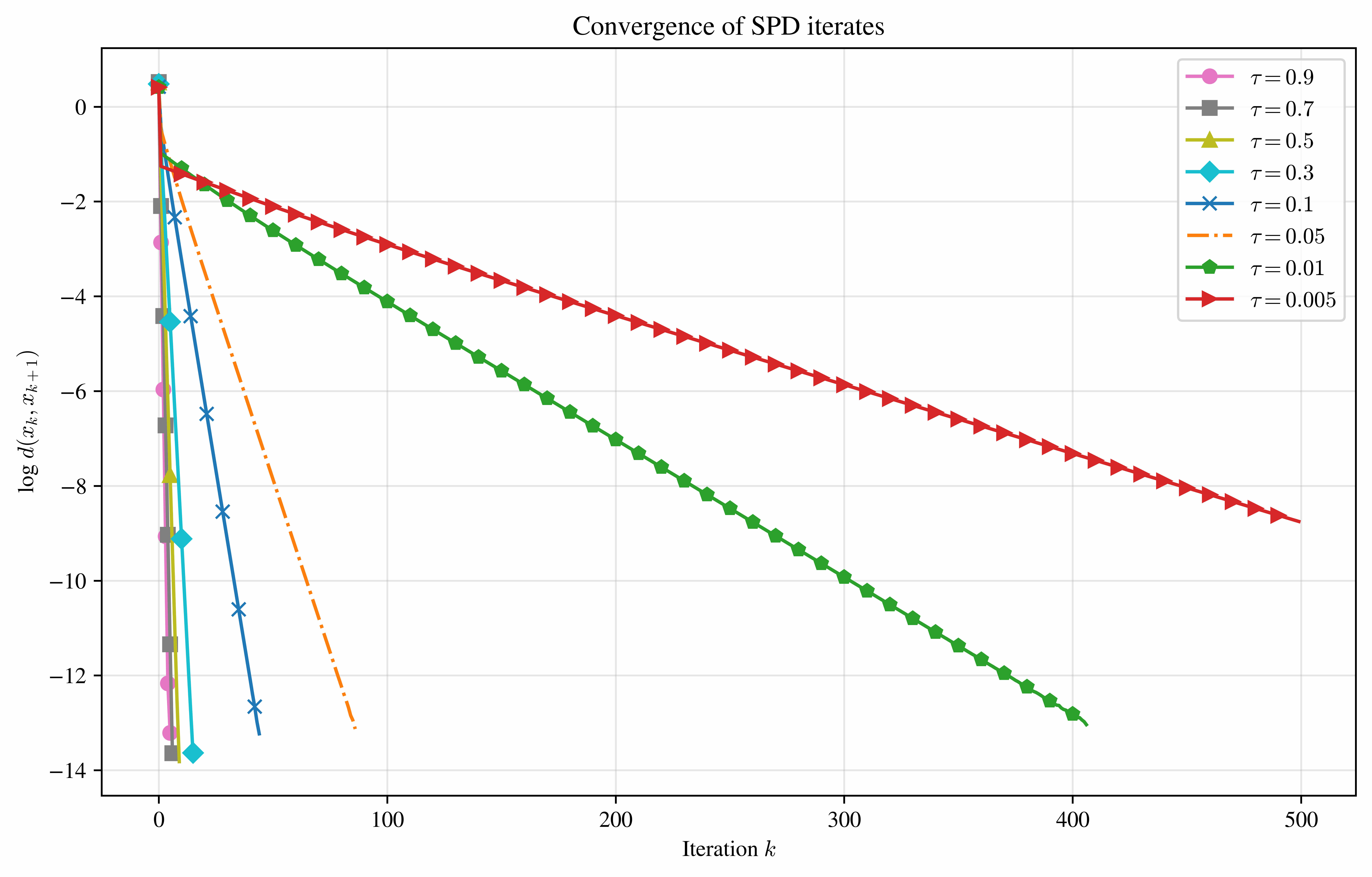}
\hfill    \includegraphics[width=0.48\linewidth]{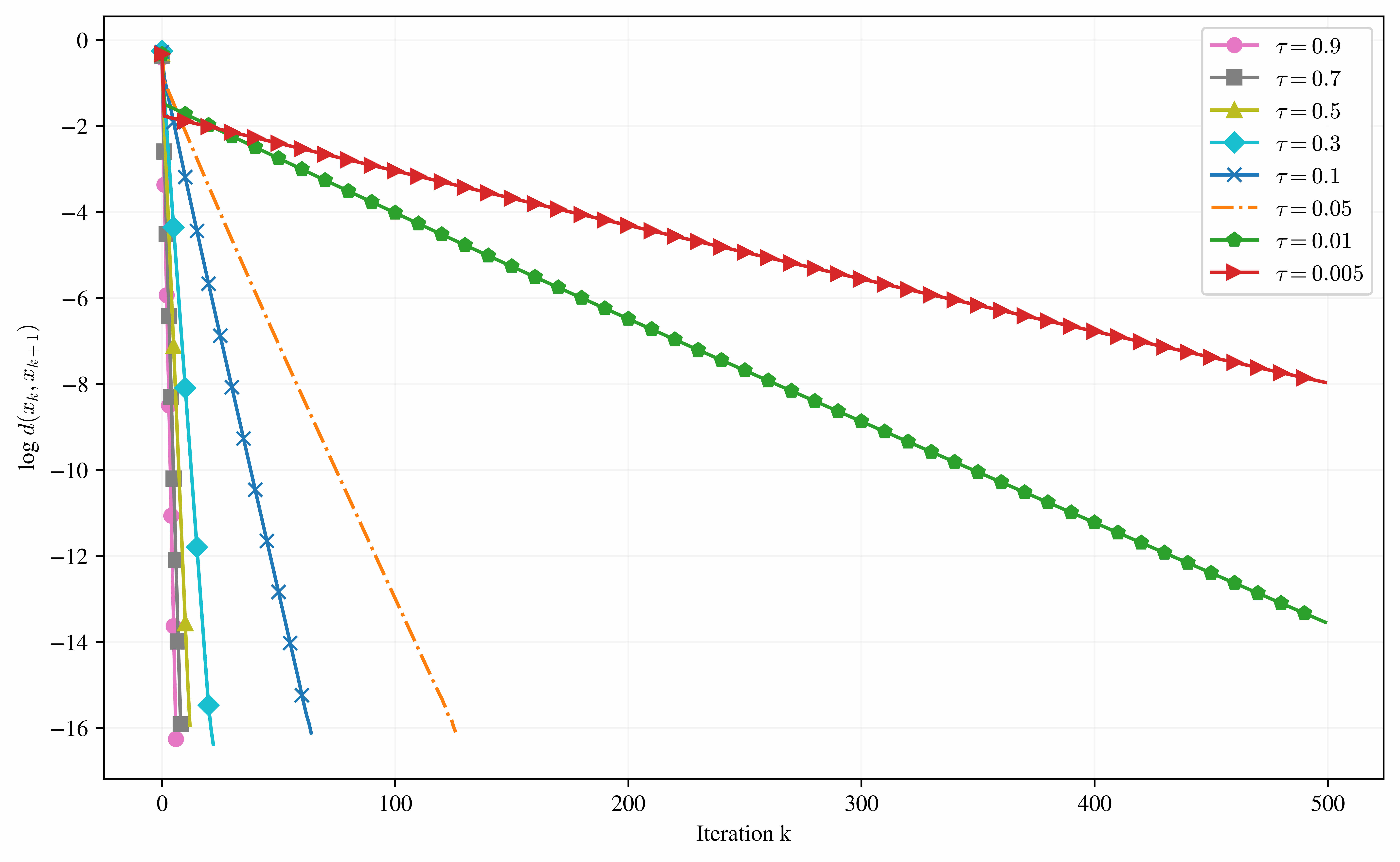}
  \caption{Convergence of Algorithm \ref{algo:cyclic_gd} for different damping parameters $\tau$
  for computatoin of the Fréchet mean on two data sets of size $20$.
  Left:  iteration on the Hadamard manifold $(\mathcal{S}_{++}^3, d)$.
  Right: iteration on the sphere $\mathbb{S}^{2}$ with the spherical metric.
  In each case, the plot shows the residual $d(x^k,x^{k+1})$ versus the iteration index $k$ on a logarithmic scale.}
  \label{fig:applications_convergence}
\end{figure}

\subsection{Spheres}

As another nonlinear model space, we consider the unit sphere
$
\mathbb{S}^{d-1}=\{x\in\mathbb{R}^d:\|x\|_2=1\},
$
equipped with the spherical (geodesic) metric
$
d_{\mathbb{S}}(x,y)=\arccos(\langle x,y\rangle),\qquad x,y\in \mathbb{S}^{d-1},
$
where $\langle\cdot,\cdot\rangle$ denotes the Euclidean inner product. Unlike the SPD manifold,
$\mathbb{S}^{d-1}$ has positive curvature. In particular, it is not globally CAT$(0)$.
However, on sufficiently small geodesically convex neighborhoods, it can be treated as a CAT$(\kappa)$
space with $\kappa>0$. By Lemma \ref{t:CATkappa-pucvx}, such neighborhoods are locally $p$-uniformly convex
with $p=2$ and a local convexity constant $c<2$. Consequently, the convergence theory of
Theorem \ref{t:msr convergence} applies whenever the corresponding assumptions are satisfied.
Optimality criteria for convex optimization on spheres using the tools of optimization on manifolds
has been explored in \cite{FerIusNem14, Ferreira2013-uq}.

For the numerical experiments, the data points on the sphere were sampled uniformly at random using
the second method in \cite{Marsa1972}, namely by generating independent standard normal vectors and
normalizing them to unit length.
The spherical experiment in Figure \ref{fig:applications_convergence}(right) exhibits behavior similar to the SPD case,
indicating that the assumptions of Theorem \ref{t:msr convergence} are satisfied.
The residual $d(x^k,x^{k+1})$ again shows an approximately linear behavior on the logarithmic scale, which reflects the
local linear convergence. It is perhaps surprising to observe that convergence appears to be globally
linear for the sphere, but this is not ruled out by the theory.  As with the Hadamard manifold
case, smaller values of $\tau$ lead to smaller
cycle diameters, and hence to more accurate
approximations of the Fréchet mean. 
For $\tau=0.9$  the maximum pairwise distance between points in the final cycle of Algorithm \ref{algo:cyclic_gd}
was $1.06131$;  for $\tau=0.005$  the maximum pairwise distance between points in the final cycle
was $0.00629994$.  The ratio of $\tau$ to the cycle diameter was not as constant over all the different values of $\tau$
(around $0.7$) indicating a slight nonlinear dependence of the cycle diameter with the relaxation parameter.

\section{Conclusion and Open Problems}
In Hadamard spaces we have shown that compositions and convex combinations of prox mappings of convex
functions are pointwise $\alpha$-fne, which together with existence and stability (Assumption \ref{ass:regularity abstr})
yields quantitative convergence of the fixed point iterations.  For the problem of computing Fr\'echet means, we believe
that Assumption \ref{ass:regularity abstr} can be removed entirely.  In other words, we conjecture that
any reasonable iteration consisting of compositions and convex combinations of prox mappings of the
distance function to the data points will converge globally linearly.  In spaces with positive curvature,
the a$\alpha$-fne property of the fixed point mapping still required additional assumptions (Assumption \ref{ass:regularity gen metric}
and \ref{ass:regularity averaged}), even when the fixed point mapping consists of compositions and convex combinations of
prox mappings of convex functions.  We conjecture that these assumptions can also be removed for
compositions and convex combinations of prox mappings of convex functions when the fixed point mapping is a self-mapping
on a space with small enough diameter,  small enough that the CAT($\kappa$) space is {\em symmetric perpendicular}
(see \cite[Remark 3, Theorem 21]{LauLuk21}).  We also believe that for the problem of computing
Fr\'echet means, Assumption \ref{ass:regularity abstr} can likewise
be removed entirely on CAT($\kappa$) spaces with small enough diameter.  In other words, we conjecture that
the iterates of any reasonable iteration consisting of compositions and convex combinations of prox mappings of the
distance function to data points on a symmetrically perpendicular CAT($\kappa$) space will converge linearly.  Moreover,
we see no obvious reason why this should not extend to objective functions that are finite sums of distances to geodesically convex sets with
at least one set being compact, so long as they are close enough to each other.  A final direction of future research is to determine
an easily identified class of functions, similar to semi-algebraic functions in linear spaces, for which gauge metric subregularity is
automatically satisfied.



\end{document}